\documentclass[12pt]{article}    
\usepackage{amsmath, amssymb, amsthm,pdfsync}
\usepackage{fullpage}
\usepackage{verbatim}
\usepackage{graphicx}
\usepackage{subfig, caption}
\usepackage{cancel}
\usepackage{enumerate}
\usepackage{enumitem}
\usepackage{wrapfig}
\usepackage{mathtools}
\usepackage{xcolor}
\usepackage{epigraph}
\usepackage{cleveref}
\usepackage{overpic}
\usepackage{esint}
\usepackage{dsfont}
\usepackage{color}
\usepackage{amsrefs}
\usepackage{mathrsfs}
\usepackage{stmaryrd}
\usepackage[all]{xy}
\usepackage[mathcal]{eucal}
\usepackage{csquotes}
\usepackage[normalem]{ulem}
\usepackage{dsfont}

\definecolor{darkgreen}{rgb}{.25,.65,.25}

\renewcommand{\phi}{\varphi}

\setlist[itemize]{itemsep=-1mm}

\newtheorem{theorem}{Theorem}
\newtheorem{lemma}[theorem]{Lemma}
\newtheorem{lem}[theorem]{Lemma}
\newtheorem{prop}[theorem]{Proposition}

\newtheorem{proposition}[theorem]{Proposition}
\newtheorem{definition}[theorem]{Definition}
\newtheorem{remark}[theorem]{Remark}

\newcommand{\vertiii}[1]{{\left\vert\kern-0.25ex\left\vert\kern-0.25ex\left\vert #1 
    \right\vert\kern-0.25ex\right\vert\kern-0.25ex\right\vert}}
\newcommand{\e}{\epsilon}

\newcommand{\R}{\mathbb{R}}

\newcommand{\eps}{\varepsilon}
\def\epsilon{\varepsilon}
\def\tilde{\widetilde}

\renewcommand{\1}{\mathds{1}}

\newcommand{\be}{\begin{equation}}
\newcommand{\ee}{\end{equation}}

\DeclareMathOperator{\id}{Id}

\DeclareMathOperator{\sign}{sign}

\numberwithin{equation}{section}
\numberwithin{theorem}{section}

\Crefname{assumption}{Assumption}{Assumptions}
\Crefname{theorem}{Theorem}{Theorems}
\Crefname{lem}{Lemma}{Lemmas}
\Crefname{cor}{Corollary}{Corollaries}
\Crefname{prop}{Proposition}{Propositions}
\Crefname{theorem}{Theorem}{Theorems}
\Crefname{conjecture}{Conjecture}{Conjectures}

\begin{document}

\title{{\bf{Slow and fast minimal speed traveling waves of the FKPP equation with chemotaxis}}}

\author{Christopher Henderson\footnote{The University of Arizona, Department of Mathematics, Tucson, AZ, USA\newline \indent\indent\hspace{-0.03in}{\em E-mail address}: ckhenderson@math.arizona.edu}}

\date{}

\maketitle

\begin{abstract}
We examine a general model for the Fisher-KPP (FKPP) equation with nonlocal advection.  The main interpretation of this model is as describing a diffusing and logistically growing population that is also influenced by intraspecific attraction or repulsion.  For a particular choice of parameters, this specializes to the Keller-Segel-Fisher equation for chemotaxis.  Our interest is in the effect of chemotaxis on the speed of traveling waves.  We prove that there is a threshold such that, when interactions are weaker and more localized than this, chemotaxis, despite being non-trivial, does not influence the speed of traveling waves; that is, the minimal speed traveling wave has speed 2 as in the FKPP case.  On the other hand, when the interaction is repulsive, we show that the minimal traveling wave speed is arbitrarily large in a certain asymptotic regime in which the interaction strength and length scale tend to infinity.
\end{abstract}



\section{Introduction and main results}\label{sec:results}

\subsubsection*{The general setting and the main goal}

The model we consider is
\begin{equation}\label{e:the_equation}
		u_t + (v u)_x = u_{xx} + u(1-u), \qquad 
			\text{ in } (0,+\infty)\times \R,
\end{equation}
where, for $\chi\in \R$, $\sigma > 0$, and $K \in L^1(\R)\cap L^\infty(\R)$, we define
\begin{equation}\label{e:advection}
	v = \chi  \, K_\sigma*u
		\qquad \text{ and } \qquad 
	K_\sigma = \frac{1}{\sigma} K\left( \frac{\cdot}{\sigma}\right).
\end{equation}
The unknown, $u$, typically represents the population density of a species that is experiencing an intraspecific aggregation or dispersion effect such as chemotaxis.  For simplicity, we always refer to this nonlocal effect as chemotaxis, although the model could be used in other settings.  The parameters $\chi$ and $\sigma$ encode the strength and length-scale of the chemotaxis, respectively, while the sign of $\chi$ determines if the effect is attractive ($\chi>0$) or repulsive ($\chi < 0$).


Before making our assumptions precise, we discuss the main goal of this article.  The equation~\eqref{e:the_equation} and similar reaction-diffusion equations model the invasion of species into a new environment.  When $\chi \equiv 0$, it is well-known that this invasion occurs with speed two.  More precisely, the minimal speed traveling wave has speed two (we define traveling waves below).  We aim to understand the effect of chemotaxis on this invasion speed; that is, we aim to answer the question: does the localized, nonlinear chemotaxis term alter speed of the minimal speed traveling wave, and, if it does, when and how is the speed changed?  While some results exist, they are either partial or require very specific structure of $K$.  We discuss this at the conclusion of this section.

\subsubsection*{Assumptions on $K$, $\chi$, and $\sigma$}

Throughout the paper, $K$ is a bounded odd function that is increasing except at $x = 0$; that is, it is increasing in $(-\infty,0)$ and in $(0,+\infty)$.  We also assume that $K$ does not change sign in $(-\infty,0)$ and in $(0,+\infty)$.  Without loss of generality, we assume that
\begin{equation}\label{e:K_normalization}
	\|K \|_{L^1(\R)} =1
		\quad \text{ and } \quad
	2 |K(0^+)|
		= - 2K(0^+)
		= 1,
\end{equation}
where $K(0^+) = \lim_{x\searrow 0} K(x)$.  Defining $K(0^-)$ analogously, notice that $2|K(0^+)| = |K(0^+) - K(0^-)|$; that is, $2|K(0^+)|$ is the magnitude of the ``jump'' of $K$ across the origin.

Finally, we assume that
\begin{equation}\label{e:K_antiderivative}
	K = \overline K'
		\qquad\text{ for some } \overline K
		\text{ such that }
		(1+|x|) \overline K \in L^1(\R).
\end{equation}
We make this assumption, but note that this is not necessary for all results below.

\begin{remark}\label{r:K}
	We give some examples of admissible choices of $K$ here.
	\begin{enumerate}[label=(\roman*)]
		\item $K(x) = -\frac{\sign(x)}{2} e^{- |x|}$.  This is the kernel related to the Keller-Segel model for chemotaxis; that is, $v = (\chi/ \sqrt \sigma) \bar v_x$ for $\bar v$ satisfying $- \sigma^2 \bar v_{xx} + \bar v = u$.  This has been studied extensively, see the discussion below.
		\item $K(x) = -\frac{\sign(x)}{2} \1_{[-1,1]}(x)$.
		\item $K(x) = -\frac{\sign(x)}{2} \left(1 + \frac{|x|}{k-1}\right)^{-k}$ for $k > 2$.
		\item $K(x) = - \frac{\sign(x)}{2} e^{- \frac{|x|^\alpha}{D_\alpha}}$ for any $\alpha \in (0,1)$, where $D_\alpha = \left(\int_0^\infty e^{-|x|^\alpha}\, dx\right)^{-\alpha}$.
	\end{enumerate}
\end{remark}

The choice of $K$ reflects modeling assumptions on intraspecific interactions.  The specific choice in \Cref{r:K}.(i) is due to an assumption that the chemical signal secreted by each individual diffuses (quickly) according the heat equation.  More broadly,~\eqref{e:the_equation} can be considered as a general model in which individuals, depending on the sign of $\chi$, either repel or attract one another.  One example of this phenomenon is herding, although many examples exist.  To further illustrate this, we discuss the choice \Cref{r:K}.(ii).  This corresponds to individuals that interact only when they are distance $\sigma$ or closer and could be due to, e.g., limited vision.  Roughly, when $\chi < 0$, they ``feel a pull'' towards one another with strength $O(|\chi|)$ with a multiplicative factor related to the fraction of the population within distance $\sigma$.  When $\chi>0$, they experience a  ``push'' instead. 
In general, aggregation-diffusion models arise both as gradient flows of a free energy involving an interaction kernel and as limits of microscopic models; see, e.g., the results, discussions, and references in~\cite{CDFLS,JabinWang}.

Our interest in this article is, roughy, in the cases $|\chi|, |\chi|/\sigma \ll 1$ and $-\chi \gg 1$.  As such, we make the standing assumption that
\begin{equation}\label{e:chi_sigma}
	\chi, \frac{\chi}{\sigma} < \frac{1}{2}.
\end{equation}
This simplifies some estimates but does not play a fundamental role in our analysis.

\subsubsection*{Traveling waves and preliminary observations}

Our main objects of study are the traveling wave solutions of~\eqref{e:the_equation}, which we define below. 
In particular, we are interested in the effect of the chemotaxis term $v$ on the speed of the traveling wave.
\begin{definition}\label{def:wave}
A traveling wave solution is a pair $(c,u)$ such that (i) $0 < u \in C^2(\R)\cap L^\infty(\R)$, (ii) $c\geq 0$, (iii) $\tilde u(t,x) := u(x-ct)$ solves~\eqref{e:the_equation}, and (iv)
\begin{equation}\label{e:wave_limits}
	\lim_{x\to-\infty} u(x) = 1,
		\quad \text{ and } \quad
		\lim_{x\to\infty} u(x) = 0.
\end{equation}
We refer to $c$ as the speed and $u$ as the profile.  We say that a traveling wave $(c^*, u^*)$ is a  minimal speed traveling wave if $c^* \leq c$ for all traveling waves $(c,u)$.
\end{definition}

For similar reaction-diffusion equations, one sometimes assumes that $u(x)$ is simply uniformly positive as $x\to-\infty$.  In fact, the convergence to one of $u$ at $-\infty$ plays almost no role in our analysis; however, it simplifies the statements of some intermediate lemmas to use the definition above.

For any traveling wave $u$, notice that
\begin{equation}\label{e:wave}
		- c u_x + (v u)_x = u_{xx} + u(1-u).
\end{equation}
This equation will be the basis for our analysis.

Using standard techniques, we can show that traveling waves exist.
\begin{proposition}\label{thm:existence_and_speed_lower}
Under the assumptions above, there exists a wave with speed
\[
	c 
		\leq 2 \sqrt{ 1 + \frac{|\chi|}{\sigma}} + \frac{|\chi|}{2}
		\leq 2 + \frac{|\chi|}{\sigma} + \frac{|\chi|}{2}.
\]
Moreover, if $(c,u)$ is any traveling wave solution of~\eqref{e:wave}, then $c\geq 2$.
\end{proposition}
We note that the lower bound on the traveling wave speed, that is, that traveling waves cannot have a {\em slower} speed than $2$, follows directly from previous results such as \cite[Theorem 1.1]{HamelHenderson}, but we provide a slightly different proof here.

\subsubsection*{Slow traveling waves}

The first main result of the article is that, under a smallness condition on $\chi\sigma^2$ and $\chi/\sigma$, the speed of the minimal speed traveling wave is exactly $2$, which is the same as in the case where chemotaxis is not present ($\chi=0$).
\begin{theorem}\label{thm:slow}
Under the assumptions above, there exists $\e_0>0$ such that, if $|\chi|(\sigma^{-1} +  \sigma^2) < \epsilon_0$, then there exists a traveling wave solution $(c,u)$ of~\eqref{e:the_equation} such that $c = 2$.
\end{theorem}

 In view of the discussion of \Cref{thm:existence_and_speed_lower}, the thrust of \Cref{thm:slow} is to establish the upper bound on $c$.  The standard method of establishing an upper bound on the speed is to build supersolutions of $u$.  There are two main obstructions to this, both stemming from the fact that $v$ depends nonlocally on $u$. First,~\eqref{e:wave} does not enjoy the maximum principle.  Second, the standard arguments use strongly imposed structure on the advection, such as periodicity, which $v$ does not have.

In some cases, such difficulties can be sidestepped by ad hoc methods; however, to our knowledge, such proofs use strongly the structure of the advection~\cite{SalakoShenXue,ConstantinKiselevRyzhik}.  These are discussed below in the discussion of the history of this problem, and, in that discussion, we describe why their methods cannot generalize to our setting.  

We now describe the main idea of the proof of \Cref{thm:slow}.  We work on the level of the ``slab problem'' on $[-a,a]$ for $a\gg1$, as is typical in the construction of traveling waves.  As described above, we need only establish an upper bound.  To this end, we treat $v$ as a given function and construct a supersolution to the linear equation
\[
	- c \overline u_x
		+ (v \overline u)_x = \overline u_{xx} + \overline u(1-u).
\]
By removing an integrating factor involving $v$, we reduce this problem to studying the principal eigenfunction and eigenvalue of
\be\label{e.c6261}
	- \phi^a_{xx} - V \phi^a = \lambda \phi^a,
		\qquad\text{where }
		V = - u - (c-2)\left(1 + \frac{c-2}{4}\right) + v\left(\frac{2c - v}{4}\right) + \frac{v_x}{2},
\ee
on $[-a,a]$ with periodic boundary conditions.  A significant advantage of~\eqref{e.c6261} is that it is self-adjoint and, thus, one may analyze $\lambda$ through the Rayleigh quotient.  There are two main steps in the analysis.

First, we establish the nonnegativity of $\lambda$.  Heuristically, this occurs because, due to the $u$ term, $V$ is negative everywhere except far to the right, where it can take, at most, a small positive value (and that can only occur if $c>2$); hence, a potential eigenfunction trying to ``minimize'' its eigenvalue, must be large only on the far right.  On the other hand, it must be periodic, meaning that it will be large on the far left where $V$ is very negative.  We are unable to make this intuition rigorous.  Instead, in order to establish this rigorously, we develop a pair of functional inequalities~\eqref{e:poincare1}-\eqref{e:poincare2} allowing us to show that, if $c > 2$,
\[
	\int (\phi_x^a)^2
		- \int V (\phi^a)^2 dx
		\geq 0,
\]
when $a$ is sufficiently large.  This, after multiplying~\eqref{e.c6261} by $\phi^a$ and integrating by parts, yields the nonnegativity of $\lambda$.

This brings us to the second step.  The fact that $\lambda\geq 0$ allows us to construct a supersolution of~\eqref{e:wave} via $\phi^a$.  This supersolution is only useful if we have good pointwise bounds on $\phi^a$ that hold for all $a$.  We control the size of $\phi^a$ by constructing barriers and using the regularity imparted by~\eqref{e.c6261}.  We use this to show that, if $c>2$, $u(0)$ must be exponentially small in the parameter $a$, which contradicts the fact that we have constructed $u$ to take a fixed positive value at $x=0$.

\Cref{thm:slow} is proved in \Cref{s:slow_tw}.

\subsubsection*{Fast traveling waves}

In view of \Cref{thm:existence_and_speed_lower} and \Cref{thm:slow}, it is natural to wonder if chemotaxis can {\em speed up} traveling waves.  Our second main result establishes this.

\begin{theorem}\label{thm:fast}
Fix any traveling wave solution $(c,u)$ of~\eqref{e:the_equation} in the sense of \Cref{def:wave}.  For any $\eps \in (0,1)$, there exists $A_\eps$ such that
\[
	c \geq (1 - \eps) \frac{|\chi|}{2}
\]
whenever
\[
	\sigma > A_\eps
		\quad\text{ and }\quad
	\frac{\sigma}{-\chi} > A_\eps.
\]
The constant $A_\eps$ depends only on $\eps$ and $K$.
\end{theorem}

\Cref{thm:fast} is essentially sharp when $1 \ll -\chi \ll \sigma$.  Indeed, the lower bound in \Cref{thm:fast} and the upper bound in \Cref{thm:existence_and_speed_lower} match up to $o(\chi)$.  Roughly, this shows that the minimal speed traveling wave has speed $|\chi|/2 + o(\chi)$.
%

We now discuss the idea of the proof.  Roughly, the wider the front\footnote{Here, we loosely think of the front as the region in which $u\not\approx 1,0$.} is in a reaction-diffusion problem, the faster the wave should propagate.  Intuitively, this is due to a wider region in which the ``reaction'' (in this case, the population reproducing) occurs.  Mathematically, one can see this by integrating~\eqref{e:wave} to obtain
\[
	c = \int_{-\infty}^\infty u(1-u)\, dx,
\]
and noticing that $u(1-u)$ is $\approx 0$ away from the front.  On the other hand, if the front is narrow, one can show that $v\approx |\chi|/2$ everywhere ``near'' the front. 
To see this, notice that, as $\sigma \gg 1$, roughly half the convolution occurs where $u \approx 1$ while the other half the convolution occurs where $u \approx 0$. In this case, the wave is ``pushed'' by $v$ at the speed $\approx |\chi|/2$.  The proof proceeds by quantifying this intuition in a careful way.

In contrast to the usual FKPP equation, in which fronts are pulled (that is, the speed is determined by the behavior far to the right where $u\approx 0$), the intuition above shows that fronts are pushed when $-\chi >2$ and $\sigma$ is sufficiently large.  Thus, the combination of \Cref{thm:slow,thm:fast} shows a transition from pulled to pushed.  To our knowledge the only similar setting in which this is observed is the Burgers-FKPP equation considered in~\cite{BramburgerHenderson}, discussed below.  In both cases, a ``threshold'' exists for certain parameters under which the advection has no effect on the speed (although the front profile is changed) and over which speed-up occurs.

\subsubsection*{History of the problem, related work, and further discussion of main difficulties}

The general model considered here was introduced by Hamel and the present author in~\cite{HamelHenderson}.  There, we considered instead the Cauchy problem, where $u$ solves~\eqref{e:the_equation} with initial data $u_0$ that is compactly supported (at least on the right), and showed that spreading speed is at least $2$.  This lower bound, which applies quite generally to a collection of models, immediately yields a lower bound for the speed of traveling waves (see \Cref{thm:existence_and_speed_lower}).  Additionally, we proved lower bounds on {\em superlinear} propagation in the case when $\chi < 0$ and $K \notin L^1$.

The majority of work to this point has been on the Keller-Segel model (see \Cref{r:K}.(i)).  While there is a large body of work regarding well-posedness on finite domains (see, e.g.,~\cite{TelloWinkler} and the literature citing it) and on propagation for the Cauchy problem, we focus our discussion here to those works investigating traveling waves and other propagation phenomena.  To our knowledge, the first work in this direction was due to Nadin, Perthame, and Ryzhik~\cite{NadinPerthameRyzhik}, who showed constructed traveling waves with (unknown) speed between $2$ and $2 + \chi/(1 - \chi/\sigma)$, in our notation, under the assumption that $0 \leq \chi < \sqrt \sigma \min\{1, \sqrt\sigma\}$.  This was later upgraded to a sharp result by Salako, Shen, and Xue~\cite{SalakoShenXue} under the additional assumptions\footnote{The model considered by Salako, Shen, and Xue has  parameters $\chi$, $a$, $b$, $\lambda$, and $\mu$; however, these can be reduced to the two parameters $\chi$ and $\sigma$ used above by scaling.} that $\sigma < 1$ and
\[
	0
		\leq \left( 1 + \frac{\sigma^{1/4} - 1}{2(\sigma^{1/4} + 1)}\right) \frac{\chi}{\sigma}
		\leq 1.
\]
Their proof is based on the ad hoc construction of sub- and supersolutions and uses strongly the structure of Keller-Segel model.  A key point in the argument is that $\overline v$ (see \Cref{r:K}) decays exponentially as $x\to\infty$.  In our setting, this is only possible if $K$ decays at least exponentially as $x\to\infty$; hence, a new argument is needed in order to include, e.g., cases (iii) and (iv) in \Cref{r:K} among others.  In some sense, their argument does not illuminate the reason for the minimal speed traveling waves to be that of the Fisher-KPP equation.  This is a partial motivation for \Cref{thm:slow}.  Specifically, our goal here is to develop {\em general} tools for proving and increase our understanding of the lack of speed-up by advection.

We briefly make note of another branch of research investigating the coupling of the Fisher-KPP equation with the {\em parabolic} Keller-Segel (as opposed to the {\em elliptic} model discussed above).  Here Salako and Shen~\cite{SalakoShenPP} have established the minimal speed as 2; see also a very elegant, simple dynamical systems argument of Bramburger~\cite{Bramburger}.  While this setting is technically not within the framework of the current manuscript, we believe the arguments used here could be extended to such a case.  Indeed, Salako and Shen's argument relies heavily on the representation formula of solution of the parabolic Keller-Segel as a convolution in space {\em and time}.  Finally, we mention that there is an ongoing program in understanding the effect of different choices of diffusion (see, e.g.,~\cite{JWXY} and references therein).

We also mention the work of Calvez~\cite{Calvez}, establishing traveling wave solutions for a mesoscopic (kinetic) model for chemotaxis not involving the logistic growth term $u(1-u)$.  Like the traveling waves in \Cref{thm:fast}, those constructed by Calvez are driven by intraspecific interactions of the population.  The differences between our model and that of Calvez are quite extreme and so the results, while quite similar in spirit, have very little connection on a technical level.

From a broader perspective, the work here contributes to the ongoing investigation of the influence of advection on front propagation in reaction-diffusion equations.  This is the other main motivation for this work.  When the advection is imposed, that is, independent of $u$, and has specific structure, implicit formulas exist~\cite{BerestyckiHamel, Xin_Book, MajdaSouganidis, EvansSouganidis}.  These are often difficult to quantify, and, to our knowledge, the only precise results are in certain asymptotic regimes such as when the advection is large~\cite{ElSmailyKirsch,HamelZlatos,NolenXin1,RyzhikZlatos,Zlatos} or small~\cite{NolenXin1,NolenXin5,HendersonSouganidis} or the diffusion is small~\cite{ElSmaily}.  We mention also the results for large cellular flows with Hamilton-Jacobi models by Xin and Yu~\cite{XinYuG,XinYuHJ} and the ABC flow by Xin, Yu, and Zlatos~\cite{XinYuZlatos}.

There are two major differences between the current setting and the results discussed in the previous paragraph: there the advection is {\em imposed} from the outside and ``uniform'' over the spatial domain, whereas the advection here is influenced by $u$ and is spatially localized near the front.  In fact, those two features are shared by many physically relevant settings (see also \cite{BerestyckiConstantinRyzhik, ConstantinKiselevRyzhik, CRRV} and the literature citing it).  Intimately tied to these features are two of the main difficulties: (1)~\eqref{e:the_equation} does not enjoy a comparison principle, meaning that techniques used in the settings described in the previous paragraph are not available, and (2) the ``advection'' $v$, in our setting, does not enjoy the specific structure that allowed precise calculations as in the works discussed in the previous paragraph.  To our knowledge, there are a dearth of results like \Cref{thm:fast} in which precise estimates are established for nontrivial, nonlinear effects coming from the advection.  We note that similar nonlinear effects on the propagation have been seen when the logistic term involves a nonlocal, nonlinear term~\cite{ABBP, BHR_nonlocal, CHMTD, Penington}, although the effects and underlying mathematical analysis are quite different.

Above we mentioned~\cite{ConstantinKiselevRyzhik} briefly, but this demands further discussion.  In this work, Constantin, Kiselev, and Ryzhik study a system composed of the Fisher-KPP equation for a quantity $T$, representing the temperature of a fluid undergoing combustion, coupled to a fluid equation such as the Navier-Stokes equation for a quantity $v$, representing the fluid velocity, with the Boussinesq approximation involving $T$ in a two dimensional cylinder.  This models the situation in which a buoyancy force arising from changes in $T$ induces a velocity field.  Heuristically, this setting is quite similar to the one considered in the current manuscript as it may be seen as a nonlocal coupling of the solution of a reaction-diffusion equation with its advection.  Constantin, Kiselev, and Ryzhik show that, under a smallness condition on the width of the cylinder and the strength of the coupling, the only traveling waves are speed 2 planar waves (it should be noted that non-planar waves exist in a variety of settings~\cite{BerestyckiConstantinRyzhik, ConstantinLewickaRyzhik, Henderson_Boussinesq, Lewicka, LewickaMucha, TexierPicardVolpert}).  This is quite similar to our \Cref{thm:slow}.  Their simple, elegant proof, however, hinges on the application of a Poincar\'e inequality on the cross-section of the cylinder (hence, the importance of the smallness of the cylinder) and is thus not applicable to the current setting.  A substantial difference, though, is that their proof hinges on showing that $v\equiv 0$, while, in our setting, $v\not\equiv 0$ and, thus, the goal is to show a {\em trivial} effect on the speed of a {\em non-trivial} advection.

One simplified, one-dimensional version of the reactive Boussinesq equation discussed above is one in which the equation for the advection is replaced by a forced Burgers equation~\cite{CRRV,BramburgerHenderson}.  In particular, in \cite{BramburgerHenderson}, Bramburger and the present author show a similar threshold: depending on a parameter representing the strength of gravity $\rho$, the minimum wave speed is $2$ if $\rho$ is below a certain critical value and grows like $\mathcal{O}(\rho^{1/3})$ otherwise.  Unfortunately, the upper and lower bounds, while explicit, match only in their order; precisely, we show a lower bound of $(3/2)^{1/3} \rho^{1/3}$ and an upper bound of $\sqrt 3 \rho^{1/3}$ as $\rho \to \infty$.  The proofs in~\cite{BramburgerHenderson} are based on transforming the Burgers-FKPP system to a system of ODE and constructing trapping regions.  This approach is not possible in the present setting as~\eqref{e:wave} cannot be converted into a system of ODE.

Finally, a recent thread of interesting work has begun on a related hyperbolic model.  Fu, Griette, and Magal~\cite{FuGrietteMagalWave} considered an inviscid version of the Keller-Segel-Fisher equation with repulsive chemotaxis; that is,~\eqref{e:the_equation} without a Laplacian, with the particular choice of $K$ given by \Cref{r:K}.(i), and with $\chi < 0$.  To our knowledge, this work and the one of Hamel and the present author~\cite{HamelHenderson} are the only works to notice the effect of repulsive chemotaxis on propagation.

In \cite{FuGrietteMagalWave}, Fu, Griette, and Magal construct a {\em discontinuous} traveling wave with speed of propagation $c\in (|\chi|/(2+|\chi|/\sigma),|\chi|/2)$, in our choice of parameters, under the assumption that $\chi/\sigma$ is positive and not too big; however, they do not obtain a general lower bound on the speed for {\em any} traveling wave as we do for the model considered here.  We note that their bounds match ours in the regime $1 \ll \chi \ll \sigma$, which is not surprising because, as our proof illuminates, propagation is entirely driven by the advection in this asymptotic regime; that is, diffusion does not contribute to the propagation.  Their proof is based on the size of the jump discontinuity in the constructed wave and does not apply to the present setting.  We note that an analogous threshold result for the asymptotic regime where $|\chi|(\sigma^{-1} + \sigma^2) \ll 1$ as in \Cref{thm:slow} is not possible in their setting as it is a consequence of the interplay between chemotaxis and diffusion.  Fu, Griette, and Magal additionally investigate the well-posedness of and convergence to $1$ of $u$ for the Cauchy problem.  We mention also their earlier works~\cite{FuGrietteMagal1, FuGrietteMagal2}. Recently, Griette, Magal, and Zhao were able to construct continuous waves for the inviscid model as well~\cite{GrietteMagalZhao}.

We briefly mention the work of Crooks as well as that of Crooks and Mascia on a class of models that depend nonlinearly on both $u$ and $u_x$, although locally~\cite{Crooks, CrooksMascia}.  Under some assumptions, they are able to analyze the speed of the minimal wave in a precise way.


%
%


\subsubsection*{Notation}

Throughout the paper, we use $C$ to denote a positive universal constant that may change line-by-line.  A universal constant is one that does not depend on $\chi$, $\sigma$, or $\tau$.  When constants have a dependence on one of these quantities, we denote it with a subscript; for example, $C_\chi$ is a constant that depends on $\chi$ but not on $\sigma$ or $\tau$.

When no confusion will arise, when using $L^p$ or other norms, we suppress the dependence on the domain; for example, we write $\|u\|_{L^\infty}$ instead of $\|u\|_{L^\infty(\R)}$.

\subsubsection*{Acknowledgements}

The author is grateful to Alexander Kiselev and Francois Hamel for helpful discussions early in the development of this work at the Banff workshop ``Interacting Particle Systems and Parabolic PDEs.''  This work was partially supported by NSF grants DMS-2003110 and DMS-2204615.

\section{Preliminaries}

We collect here some preliminary results on the size of $u$ and $v$ as well as the monotonicity of $u$.   These are used in the proof of \Cref{thm:fast}.  Since analogous results (with exactly the same proof) hold for the slab problem used in the proof of \Cref{thm:slow}, we state and prove them here and omit the proofs for the slab problem in the sequel.

\begin{lem}[Upper bound on $u$]\label{l:upper_bound}
	For any $c$, we have $\|u\|_{L^\infty} \leq \max\left\{1, \left(1 - \chi/\sigma\right)^{-1}\right\}$.
\end{lem}
This lemma is a simple consequence of the maximum principle, using that $\|v\|_{L^\infty} \leq \| u \|_{L^\infty}$ and that, for any $x$,
\begin{equation}\label{e:v_x}
	 v_x(x) = - \frac{\chi}{\sigma} u(x) + \chi \int_0^\infty (K_{\sigma})_x(y)(  u(x-y) +  u(x+y))\, dy.
\end{equation}
Hence, we omit the proof.  We note that, in the above, there is a slight abuse of notation: we do not make any assumptions on the regularity of $K_\sigma$, so $(K_\sigma)_x$ should be understood in the sense of Radon measures\footnote{That is, $\mu = dK_\sigma$ is the measure on $[0,\infty)$ such that $\mu((a,b)) = K_\sigma(b) - K_\sigma(a)$ for any $0 \leq a < b$.}.

We observe some preliminary bounds on $v$ and $v_x$ that are deduced directly from the definition of $K$ and the maximum principle:
\begin{lemma}\label{l:convolution_bounds}
For any $c$, we have
\[
\begin{split}
	&\|v\|_{L^\infty}
		=|\chi|\, \|K_\sigma *u\|_{L^\infty}
		\le \frac{|\chi|}{2}\|u\|_{L^{\infty}}
	\qquad \text{ and }\\
	&\|v_x\|_{L^\infty}
		= |\chi|\, \|(K_\sigma*u)_x\|_{L^{\infty}}
		\le \frac{|\chi|}{\sigma}\,\|u\|_{L^{\infty}}.
\end{split}
\]
\end{lemma}
Both inequalities in \Cref{l:convolution_bounds} follow from Young's inequality for convolutions.  The proof of the first inequality uses explicitly the cancellations due to the positivity properties of $K$ and $u$ in order to pick up the extra factor of $1/2$.  We omit the proof.

A useful tool is that $u$ is monotonic when it is small.  

\begin{lemma}[Monotonicity of $u$]\label{l:monotonicity}
	Suppose that $u$ is a traveling wave solution of~\eqref{e:the_equation} in the sense of \Cref{def:wave}.  Then $u_x \leq 0$ whenever:
	\[\begin{split}
		&u < \frac{1}{1- \frac{\chi}{2\sigma}}
			= \frac{1}{1 + \frac{|\chi|}{2\sigma}}
		\quad\text{if}\quad
		 \chi \leq 0
		 \qquad\text{ or } \qquad
		u < \frac{1 - \frac{2\chi}{\sigma}}{\left(1 - \frac{\chi}{\sigma}\right)^2}       
		\quad\text{if}\quad
		\chi > 0.
	\end{split}\]
\end{lemma}
\begin{proof}
Fix any $\epsilon>0$.  Let $\overline u_{\rm mon}$ be $(1-\chi/2\sigma)^{-1}$ if $\chi \leq 0$ and $(1 - 2\chi/\sigma)/(1-\chi/\sigma)^2$ if $\chi > 0$.  We argue by contradiction, assuming that there exists $x_0$ such that $u(x_0) \leq \overline u_{\rm mon}-\eps$ and $u_x(x_0) > 0$.   Define
\[
	y_0 =
		\min\left\{ y : u(y) \text{ is a local minimum and } u(y) \leq \overline u_{\rm mon}-\eps\right\}.
\]
Since $u \to 1$ as $x\to-\infty$ and due to the existence of $x_0$, it follows that $y_0$ is well-defined.  We point out that $u(x) \geq u(y_0)$ for all $x <y_0$ (if this were not true, we could find another point in the set defining $y_0$ that is smaller than $y_0$).

We re-write~\eqref{e:wave} as
\[
	 u_{xx}
		- \left(v - c\right) u_x
		= - u\left(1 - u - v_x\right).
\]
We claim that $1 - u - v_x > 0$ at $y_0$.  If this were true, then, since $y_0$ is the location of a local minimum, we find
\[
	0
		\leq u_{xx} - \left(v - c\right) u_x
		= - u \left(1 - u - v_x\right)
		< 0,
\]
which is a contradiction.  Hence, due to this and the arbitrariness of $\eps$, the proof is finished if we can establish the positivity of $1 - u - v_x$.

We now establish this inequality.  Using~\eqref{e:v_x}, we notice that, at $y_0$
\begin{equation}\label{e.c6231}
\begin{split}
	1 - u - v_x
		&= 1 - u \left(1 - \frac{\chi}{\sigma}\right)
			- \chi \int_0^\infty (K_\sigma)_y(y) (u(y_0 - y) + u(y_0 +y))\, dy.
\end{split}
\end{equation}
When $\chi \leq 0$, $-\chi (K_\sigma)_x\geq 0$.   Hence, using the choice of $y_0$, we find
\[\begin{split}
	- \chi \int_0^\infty (K_\sigma)_y(y) (u(y_0 - y) + u(y_0 +y))\, dy
		&\geq - \chi \int_0^\infty (K_\sigma)_y(y) u(y_0 - y)\, dy\\
		&\geq - \chi u(y_0) \int_0^\infty (K_\sigma)_y(y)\, dy
		= \frac{|\chi|}{2\sigma} u(y_0).
\end{split}\]
Using this along with the definition of $\overline u_{\rm mon}$ in~\eqref{e.c6231}, we find, at $y_0$,
\[
	1 - u - v_x
		\geq 1 - u \Big(1 + \frac{|\chi|}{2\sigma}\Big)
		\geq 1 - (\overline u_{\rm mon} -\eps) \Big(1 + \frac{|\chi|}{2\sigma}\Big)
		> 0.
\]
The proof is thus finished in this case.  

When $\chi >0$, $-\chi (K_\sigma)_y \leq 0$.  Hence, by \Cref{l:upper_bound},
\[
	-\chi \int_0^\infty (K_\sigma)_y(y) (u(x - y) + u(x +y))\, dy
		\geq -2 \|u\|_{L^\infty} \chi \int_0^\infty (K_\sigma)_y(y)\, dy
		\geq -\frac{\chi}{\sigma} \frac{1}{1-\frac{\chi}{\sigma}}.
\]
Including this inequality in~\eqref{e.c6231} and using the definition of $\overline u_{\rm mon}$, we find
\[
	1 - u(x) - v_x(x)
		> 1 - (\overline u_{\rm mon} -\eps) \left(1 - \frac{\chi}{\sigma}\right)
			- \frac{\chi}{\sigma} \frac{1}{1 - \frac{\chi}{\sigma}}
		> 0.
\]
In both cases, we have established that $1 - u - v_x > 0$, which completes the proof.
\end{proof}

Finally we recall that $u$ converges to $1$ on the left when $\chi/\sigma$ is appropriately bounded above.  We cite~\cite{HamelHenderson} here as it works out-of-the-box; however, the original result is~\cite[Theorem C]{SalakoShenPE}, which is stated only for the Keller-Segel model but whose proof works analogously in the more general case.  We note that the proof of the result below requires~\eqref{e:chi_sigma} and it may not be true without this assumption.

\begin{lemma}\label{l:one}\cite[Corollary~2.3]{HamelHenderson}
Suppose that $(c,u)$ is a traveling wave solution of~\eqref{e:the_equation} in the sense that $\lim_{x\to\infty} u(x) = 0$, $c>0$, and $(c,u)$ satisfies~\eqref{e:wave}.  Then $\lim_{x\to-\infty} u(x) = 1$.
\end{lemma}

\section{Speed 2 traveling waves: \Cref{thm:existence_and_speed_lower} and \Cref{thm:slow}}

In this section, we construct traveling waves whose speed is unperturbed off the case $\chi=0$ when $\chi(\sigma^2 + \sigma^{-1})$ is sufficiently small.  We begin by constructing traveling waves in the usual fashion; that is, building solutions to the problem on a finite domain (a ``slab'') via the Leray-Schauder fixed point theorem and then taking a limit as this approximating domain tends to $\R$.  This is the proof of \Cref{thm:existence_and_speed_lower}.  

After, we return to this slab problem to analyze the bounds on $c$, from which we obtain the desired upper bound.  This is the proof of \Cref{thm:slow}.  We separate the precise bounds into another step for clarity since its proof is quite technical and is independent of the rest of the construction.

\subsection{The problem on a finite slab}

Let $a>0$, $\tau \in [0,1]$, and $\theta \in (0,\theta_0)$ for
\be\label{e.theta_0}
	\theta_0 < \min\Big\{ \frac{1}{100}, \frac{1- \frac{2|\chi|}{\sigma}}{1+ \frac{|\chi|}{\sigma}}\Big\}
\ee
that is chosen in the sequel.  
Throughout this section $\theta_0$ is a constant decreased in order to encode the various statements ``for $\theta$ sufficiently small'' below (the notation is introduced in case the upper bound $\theta_0$ of $\theta$ must be referenced).  The parameter $\tau$ is for the fixed point theorem referenced above.

Consider
\begin{equation}\label{e:slab}
	\begin{cases}
		- c u_x + \tau (\tilde v u)_x = u_{xx} + u(1-u), \qquad &\text{ in } (-a,a),\\
		u(-a) = 1, \quad u(a) = 0, \quad \max_{x \geq 0} u(x) = \theta,
	\end{cases}
\end{equation}
where we define $\tilde v = \chi K_\sigma * \tilde u$ with
\begin{equation}\label{e.u_tilde}
	\tilde u(x) =
			\begin{cases}
				1, \quad &\text{ if } x \leq -a,\\
				u(x), &\text{ if } x \in (-a,a),\\
				0, &\text{ if } x \geq a.
			\end{cases}
\end{equation}
It is apparent that were $u$ and $c$ to exist, both would depend on $a$ and $\theta$.  We suppress that dependence in the notation in this subsection.  In the sequel, however, we denote $u=u^a$ and $c=c^a$ in order to clarify the presentation of the limit $a\to\infty$.

Our goal is to produce a solution to $(c,u)$ to~\eqref{e:slab} when $a$ is sufficiently large and $\theta$ is sufficiently small.  The first step is obtaining {\em a priori} estimates on any solution $(c,u)$ of~\eqref{e:slab}.

\subsubsection{Bounds on $u$}

It is clear that the maximum principle yields the same upper bound of $u$ in \Cref{l:upper_bound} and the same monotonicity properties as \Cref{l:monotonicity}.  We collect this in the following (recall~\eqref{e.theta_0}):
\begin{lem}[Upper bound on and monotonicity of $u$]\label{l:upper_bound_a}
		For all $a>0$ and $\theta \in (0,\theta_0)$, we have
		\[
			\|u\|_{L^\infty} \leq \max\left\{1, \left(1 - \chi/\sigma\right)^{-1}\right\}
				\qquad\text{and}\qquad
			u_x \leq 0 \ \text{ on } [0,a].
		\]
\end{lem}

In addition, a simple computation using the definition of $\tilde v$ yields
\begin{equation}\label{e:v_bound}
	\|\tilde v\|_{L^\infty}
		\leq \frac{|\chi|}{2} \max\{1, (1 - \chi/\sigma)^{-1}\}
	\quad \text{ and } \quad
	\|\tilde v_x\|_{L^\infty}
		\leq \frac{|\chi|}{\sigma}\max\{1, (1-\chi/\sigma)^{-1}\}.
\end{equation}
See \Cref{l:convolution_bounds} for similar bounds.

Next we get a preliminary, suboptimal bound on the front speed that is sufficient for the fixed point argument.  

\begin{lem}[Upper bound on the speed]\label{l:speed_upper_bound}
	There exists $\theta_0>0$ and $C>0$ such that if $\theta \in (0,\theta_0)$ and $a$ is sufficiently large, depending only on $\theta$, then
	\[
		c
			\leq 2\sqrt{1 + \frac{\tau|\chi|}{\sigma}} + \frac{\tau|\chi|}{2}
			\leq 2
			+ \frac{\tau |\chi|}{\sigma}
			+ \frac{\tau |\chi|}{2}.
	\]
\end{lem}
\begin{proof}
	Let $\psi_A(x) = A e^{-\lambda x}$ for $\lambda >0$ to be chosen.  Let $A_0 = \min\{A \in \R : \psi_A \geq u\}.$ 
	First, notice that this is well-defined since  $A = e^{\lambda a} \|u\|_{L^\infty}$ is in the set, while no negative $A$ are in the set due to the nonnegativity of $u$.  In fact,  $A_0\geq \theta>0$ since
	\[
		u(0) = \theta 
			\quad\text{ and }\quad
		\psi_A(0)
			= A.
	\]
	Finally, by continuity, there exists a point $x_0\in [-a,a]$ such that $\psi_{A_0}(x_0) = u(x_0)$.

	Since $\psi_{A_0} > 0$ and $u(a) = 0$, then $x_0 \in [-a,a)$.  In addition, if $x_0 = -a$, then $1 = u(-a) = \psi_{A_0}(-a) = A_0 e^{\lambda a}$.  Hence,
	\[
		\theta
			= u(0)
			\leq \psi_{A_0}(0)
			= A_0
			= e^{-\lambda a}.
	\]
	This cannot happen if $a > \lambda^{-1}\log(1/\theta)$.
	
	Thus, we have that $x_0 \in (-a,a)$.  Then $\psi_{A_0} - u$ has an interior minimum at $x_0$ of zero, which yields
	\[
		u(x_0) = \psi_{A_0}(x_0),
		\quad
		\psi_{A_0,x}(x_0) = u_x(x_0), 
		\quad\text{ and }\quad
		\psi_{A_0,xx}(x_0)\geq u_{xx}(x_0).
	\]
	Using~\eqref{e:slab}, these identities, \eqref{e:v_bound}, and~\eqref{e:chi_sigma}, we find, at $x_0$,
	\[\begin{split}
		0
			&= u_{xx} + u(1-u) + c u_x - \tau(\tilde v u)_x
			\leq \psi_{A_0}\left( \lambda^2 + 1 - c\lambda +\lambda \tau \tilde v - \tau \tilde v_x\right)\\
			&\leq \psi_{A_0}\left( \lambda^2 + 1 - c\lambda + \frac{\tau |\chi| \lambda}{2} + \tau \frac{|\chi|}{\sigma}\right).
	\end{split}\]
	
	We now choose $\lambda = (c- \tau |\chi|/2)/2$ to find, using the positivity of $\psi_{A_0}$,
	\[\begin{split}
		0
			&\leq - \frac{(c - \tau |\chi|/2)^2}{4}
				+ 1 + \tau \frac{|\chi|}{\sigma}.
	\end{split}\]
	This implies that $c \leq \tau |\chi|/2 + 2 \sqrt{1 + \tau |\chi|/\sigma}$, which, after an application of Taylor's theorem, concludes the proof.
\end{proof}

We now know that all coefficients in our equation are bounded in $L^\infty$, which allows us the use of the Harnack inequality.  We use this to derive an estimate quantifying the fact that if $u$ is ``small'' somewhere then it is small ``nearby.''

\begin{lem}[Upper bound on spatial growth]\label{l:harnack}
	There exists a constant $C_{\chi,\chi/\sigma}$ such that
	\[
		u(x) \leq C_{\chi,\chi/\sigma} u(y) e^{C_{\chi,\chi/\sigma}|x-y|}
			\qquad\text{ for all } x,y \in [-a+1, a-1].
	\]
\end{lem}
\begin{proof}
This follows by a simple iteration of the Harnack inequality $n \approx |x-y|$ times.
\end{proof}

\begin{lem}[Lower bound on the speed]\label{l:speed_lower_bound}
	For all $\epsilon>0$, there exists $\theta_\e$ and $a_\e$ such that if $\theta \leq \theta_\e$ and $a> a_\e$ then $c\geq 2 - \epsilon$.
\end{lem}
\begin{proof}
	Here we simply construct a subsolution that ``pushes'' $u$ to speed $2$.  We argue by contradiction.  Suppose that $\e>0$ and $c < 2-\e$.
	
	For any $A, \lambda>0$ and $R\in (0,a-1]$ to be chosen, let
	\[
		\underline u_A(x)
			= \frac{1}{A} e^{-\lambda x} \cos\left(\frac{\pi}{2} \frac{2x - R}{R}\right)
			\qquad \text{ for } x\in \left[0,R\right].
	\]
	Since $u > 0$ on $[0, R]$, $\underline u_A < u$ on $[0,R]$  if $A$ is sufficiently large.  Let
	\[
		A_0 = \inf\left\{A>0 : \underline u_A < u \text{ on } \left[0,R\right]\right\}.
	\]
	It is clear that $A_0$ is well-defined and that, by continuity, there exists $x_{\text{\rm\tiny touch}} \in [0, R]$ such that $\underline u_{A_0}(x_{\text{\rm\tiny touch}}) = u(x_{\text{\rm\tiny touch}})$.  Since $u(0), u(R) > 0$ and $\underline u_{A_0}(0) = \underline u_{A_0}(R) = 0$, we find $x_{\text{\rm\tiny touch}} \in (0,R)$.  Finally, by construction $u - \underline u_{A_0} \geq 0$ on $[0,R]$ and has a minimum over $[0,R]$ of $0$ at $x_{\text{\rm\tiny touch}}$.

	For notational ease, we drop the $A_0$ notation and denote $\underline u_{A_0} = \underline u$.  
	Since $u - \underline u$ has a minimum of $0$ at $x_{\text{\rm\tiny touch}}$, we find
	\[
		u_x(x_{\text{\rm\tiny touch}}) = \underline u_x(x_{\text{\rm\tiny touch}})
			\qquad\text{ and }\qquad
		u_{xx}(x_{\text{\rm\tiny touch}}) \geq \underline u_{xx}(x_{\text{\rm\tiny touch}}).
	\]
	Using these along with~\eqref{e:slab}, we find, at $x_{\text{\rm\tiny touch}}$,
	\begin{equation}\label{e.c3}
	\begin{split}
		0 &\leq ( u- \underline u)_{xx}
			= - c u_x + \tau \left( \tilde v u\right)_x - u(1-u)
				- \underline u_{xx}\\
			&= - c \underline u_x + \tau \tilde v_x \underline u + \tau \tilde v \underline u_x - \underline u(1- u)
				- \underline u_{xx}.
	\end{split}
	\end{equation}
	Recall that $u$ is decreasing on $[0,a]$, and, hence, $\underline u_x(x_{\text{\rm\tiny touch}}) = u_x(x_{\text{\rm\tiny touch}}) \leq 0$.  In addition, $u(x_{\text{\rm\tiny touch}}) \leq u(0) = \theta$.

	We first estimate $\tilde v$ and $\tilde v_x$.  Fix $L>0$ to be chosen.  We claim that there is a constant $C_{\chi,\sigma}$ such that
	\begin{equation}\label{e.c51604}
		|\tilde v_x(x_{\text{\rm\tiny touch}})|, |\tilde v(x_{\text{\rm\tiny touch}})|
			\leq C_{\chi,\sigma} K_\sigma(L)  + C_{\chi,\sigma} \int_L^\infty |K_\sigma(x)|dx + C_{\chi,\sigma} \theta e^{C_{\chi,\sigma} L}.
	\end{equation}
	We show the argument only for $\tilde v$.  The argument is similar for $\tilde v_x$, though additionally using~\eqref{e:v_x}.  Indeed, decomposing the integral and then using \Cref{l:harnack} on the close-to-$x_{\rm touch}$ portion, we find
	\[
		\begin{split}
		|\tilde v(x_{\rm touch})|
			&\leq \int_{-L}^L |K_\sigma(y)| u(x_{\rm touch}-y) dy
				 + \left(\int_{-\infty}^{-L} + \int_L^\infty\right) |K_\sigma(y)|u(x_{\rm touch}-y) dy
			\\&
			\leq C_{\chi, \chi/\sigma}\int_{-L}^L |K_\sigma(y)| u(x_{\rm touch}) e^{C_{\chi, \chi/\sigma} |y|} dy
				+ \|u\|_{L^\infty}\left(\int_{-\infty}^{-L} + \int_L^\infty\right) |K_\sigma(y)| dy.
		\end{split}
	\]
	Then~\eqref{e.c51604} is obtained by using the following: $K_\sigma$ is bounded and odd, $u(x_{\rm touch}) \leq \theta$, and $\|u\|_{L^\infty}$ is bounded due to \Cref{l:upper_bound_a}.

Returning to~\eqref{e.c51604}, choosing $L$ sufficiently large and $\theta$ sufficiently small, we have $|\tilde v_x(x_{\text{\rm\tiny touch}})|, |\tilde v(x_{\text{\rm\tiny touch}})| < \e^4$.  Using this, as well as the facts that $\tau \leq 1$, $c < 2-\eps$, $\underline u_x \leq 0$, and $u\leq \theta$, we find
	\[
		0 
			\leq - \left(2-\e + \e^4\right) \underline u_x - \underline u(1- \theta - \e^4)
				- \underline u_{xx}. 
	\]

%
%
%
%
%
%
%
%

	Let $y_{\text{\rm\tiny touch}} = \pi (2x_{\text{\rm\tiny touch}} -R)/(2R)$.  Using the definition of $\underline u$ to compute $\underline u_x$ and $\underline u_{xx}$ and dividing through by $e^{-\lambda x_{\text{\rm\tiny touch}}}/A$, we find
	\[\begin{split}
		&0 
			\leq \left(2 - \e  + \e^4\right) \left( \frac{\pi}{R} \sin\left(y_{\text{\rm\tiny touch}}\right) + \lambda \cos\left(y_{\text{\rm\tiny touch}}\right)\right)
			- \cos(y_{\text{\rm\tiny touch}})(1-\theta - \e^4)\\
			&\qquad 
				- \left(- \frac{\pi^2}{R^2}\cos(y_{\text{\rm\tiny touch}}) + 2\lambda \frac{\pi}{R} \sin(y_{\text{\rm\tiny touch}})+ \lambda^2\cos(y_{\text{\rm\tiny touch}})\right)\\
			&= \cos(y_{\text{\rm\tiny touch}}) \left( \left(2 - \e + \e^4\right) \lambda  - 1 +\theta + \e^4 - \lambda^2 + \frac{\pi^2}{R^2}\right)
				+ \sin(y_{\text{\rm\tiny touch}})\left(\left(2 - \e + \e^4 - 2 \lambda\right)\frac{\pi}{R}\right).
	\end{split}\]
	To eliminate the $\sin$ term, whose sign is unknown, let $\lambda = 1 - \e/2 + \e^4/2$.  This yields
	\[
		0
			\leq \cos(y_{\text{\rm\tiny touch}}) \left( \frac{1}{4} \left(2 - \e + \e^4 \right)^2 - 1 + \theta + \e^4 + \frac{\pi^2}{R^2}\right).
	\]
	Further increasing $R$ (and, thus, $a$) and decreasing $\theta$ and $\e$, if necessary, the coefficient of $\cos$ is negative.  In addition, as $x_{\text{\rm\tiny touch}} \in (0,R)$, $y_{\text{\rm\tiny touch}} \in (-\pi/2,\pi/2)$, and, hence, $\cos(y_{\text{\rm\tiny touch}}) > 0$.  We conclude that the the right hand side above is negative, which is a contradiction.  The proof is finished.
\end{proof}

\subsubsection{Existence of a solution on the slab}

We now use Leray-Schauder degree theory to establish the following proposition.  See~\cite{Nirenberg_book} for a review of standard results in this theory that are used below.

\begin{prop}\label{p:slab}
	There exists $\theta_0>0$ such that, for any $\theta \in (0,\theta_0)$ and any $a$ sufficiently large, there exists a solution $(c,u)$ of~\eqref{e:slab} with $\tau =1$.  Further, $u \geq \theta$ on $[-a,0]$.
\end{prop}
\begin{proof}
We first point out the following bound on any solution $(c,u)$ of~\eqref{e:slab}.  By \Cref{l:upper_bound_a,l:speed_lower_bound,l:speed_upper_bound} and elliptic regularity theory, we have that, for any $\alpha \in (0,1)$, if $\theta$ is sufficiently small and $a$ sufficiently large (depending only on $\theta$), there exists a constant $C_0>0$, depending on $\chi$ and $\sigma$ but not on $\tau$, such that
\begin{equation}\label{e:a_priori_tw}
	\frac{1}{C_0}
		< c
		\quad\text{ and }\quad
	 c + \|u\|_{C^{2,\alpha}}
		\leq C_0.
\end{equation}

Let $\mathcal{B} = \{(c,u) \in [0,2C_0]\times C^{1,\alpha} : \|u\|_{C^{1,\alpha}} \leq 2C_0, u \geq 0\}$.  Consider the operator:
\[
	S_\tau : \mathcal{B} \to C^{1,\alpha} \times \R
\]
such that $S_\tau(c,u) = (c + \theta - \max_{x\geq 0} u(x), \bar u)$, where $\bar u$ is the unique solution of the linear problem
\[\begin{cases}
	-c \bar u_x + \tau ( \chi (K_\sigma * \tilde u) \bar u)_x = \bar u_{xx} + u(1- u) \quad & \text{ for all } x \in (-a,a),\\
	\bar u(-a) = 1 \quad \bar u (a) = 0
\end{cases}\]
(recall the definition of $\tilde u$ from~\eqref{e.u_tilde}).   It is important to note that $\tilde u$ is defined in terms of $u$, not $\bar u$, so that the above equation is linear in $\bar u$.

At this point we note that the existence of a solution of~\eqref{e:slab} is equivalent to the existence of a fixed point of $S_1$.  We now establish the existence of such a fixed point.

Standard elliptic regularity theory provides bounds on the $C^{2,\alpha}$ norm of $\bar u$ depending only on the $C^{1,\alpha}$ norm of $u$ and $a$.  Hence, we have that $S_\tau$ is a compact operator.  Moreover, if $a$ is sufficiently large, any fixed point of $S_\tau$ is an element of the interior of $\mathcal{B}$, by~\eqref{e:a_priori_tw}.  Hence, we have that
\[
	\deg(\id - S_1, \mathcal{B}, 0) = \deg(\id - S_0, \mathcal{B}, 0).
\]
On the other hand, any fixed point of $S_0$ is a solution to the ``slab'' problem for the Fisher-KPP equation.  The uniqueness of such solutions is well-known (and easy to establish).  Hence, $\deg(\id - S_0, \mathcal{B}, 0) = 1$ or $-1$.  In either case, we see that $\deg(\id - S_1, \mathcal{B}, 0) \neq 0$, which establishes the existence of a fixed point.

The last step is to establish the lower bound of $u$ on $[-a,0]$.  To this end, if $u(x_0) < \theta$ for some $x_0 \in [-a,0]$, then we may apply \Cref{l:monotonicity} (which can clearly be extended to the slab problem) to conclude that $u$ is decreasing on $[x_0,a]$.  It follows that $u(0) \leq u(x_0) < \theta$, which contradicts the fact that, by construction, $u(0) = \theta$.  We conclude that $u\geq \theta$ on $[0,a]$, as claimed.  Thus, the proof is complete.
\end{proof}

\subsection{Taking $a\to\infty$: the proof of \Cref{thm:existence_and_speed_lower}}

\begin{prop}\label{prop:tw}
There exists a traveling wave solution $(c,u)$, in the sense of \Cref{def:wave} with
\[
	c
		\leq 2 \sqrt{ 1 + \frac{|\chi|}{\sigma}} + \frac{|\chi|}{2}
		\leq 2 + \frac{|\chi|}{\sigma} + \frac{|\chi|}{2}.
\]
Furthermore, $u(-\infty) = 1$.
\end{prop}
\begin{proof}
For clarity, we denote the solution constructed in~\eqref{p:slab} as $(c^a, u^a)$.  The bounds in~\eqref{e:a_priori_tw} are uniform in $a$.  Hence, for any $\theta \in (0,\theta_0)$, there exists a sequence $a_n \to \infty$ and $(c,u) \in [0, 2 + |\chi|/2 + |\chi|/\sigma]\times C^{2,\alpha}$ such that, as $n \to \infty$, $c^{a_n} \to c$ and $u^{a_n} \to u$.  Using~\eqref{e:slab}, it is easy to check that
\begin{equation}\label{e.c4}
	- c u_x + (v u)_x = u_{xx} + u(1-u) \qquad \text{ for all } x\in \R,
\end{equation}
where $v$ is as in~\eqref{e:advection}.

Next, we establish the behavior of $u$ as $x\to\infty$.  By \Cref{l:monotonicity}, let $\delta = \lim_{x\to\infty} u(x)$. Let $u_n(x) = u(x + n)$.  Since $u_n$ satisfies~\eqref{e:a_priori_tw}, we may take a subsequence $n_k \to \infty$ such that $u_{n_k} \to \delta$ locally uniformly in $C^2$.  Using this, as well as the fact that $u_n$ satisfies~\eqref{e.c4} with $v_n(x) = v(x+n)$, we find
\[
	0
		= \lim_{n\to\infty} \left(- c (u_n)_x + (v_n u_n)_x - (u_n)_{xx} - u_n(1-u_n)\right)
		= \delta(1-\delta).
\]
Note that we used that, since $u_n \to \delta$ locally uniformly then $v_n \to 0$ locally uniformly.  As $\delta < \theta < 1$, it follows that $\delta = 0$.  Thus $\lim_{x\to\infty} u(x) = 0.$

It follows immediately from \Cref{l:monotonicity} and the fact that $u(0) = \theta$ that $\liminf_{x\to-\infty} u(x) \geq \theta > 0$.  The fact that this limit is equal to one is due to \Cref{l:one}.  This concludes the proof.
\end{proof}

In addition, we can show that all traveling waves, including the one constructed in \Cref{prop:tw} have speed at least 2.  One could prove this in the manner of \cite[Proposition~4.1]{NadinPerthameRyzhik} or by changing variables to the time-dependent problem and applying \cite[Theorem~1.1]{HamelHenderson}; however, we provide another proof here in the interest of giving another perspective.

\begin{proposition}\label{prop:c_bigger_than_2}
	Suppose that $u$ is a traveling wave solution of~\eqref{e:wave} as in \Cref{def:wave}.  Then $c\geq 2$.
\end{proposition}
\begin{proof}
	The proof of this claim follows exactly as in the \Cref{l:speed_lower_bound} with the following modification.  Instead of constructing the subsolution $\underline u$ on $[0,R]$, one constructs the subsolution shifted by $L$ to be on $[L,L+R]$ for $L\gg 1$.  Then all $\theta$ in the proof, which are used to bound $u(x_{\text{\rm\tiny touch}})$, become $\theta_L = \min_{[L,\infty)} u$.  The proof follows since $\lim_{L\to\infty} \theta_L = 0$ by~\eqref{e:wave_limits}.
\end{proof}

Clearly the combination of \Cref{prop:tw} and \Cref{prop:c_bigger_than_2} yield \Cref{thm:existence_and_speed_lower}.

\subsection{Minimal speed traveling waves with speed $2$}\label{s:slow_tw}



For notational ease, we fix a sequence $a_n \to \infty$ along which $(c_n, u_n) = (c^{a_n}, u^{a_n})$ converges to $(c_U, U)$, where we use $(c_U, U)$ to denote the traveling wave constructed in the previous section.  Since we mainly work with $u_n$ in this section, we suppress the dependence on $n$; however, $u$ always refers to the approximate traveling wave on the slab $[-a_n, a_n]$ and $U$ refers to its locally uniform limit as $n\to\infty$.  We suppress the dependence on $\chi$ and $\sigma$ when no confusion will arise; otherwise we denote such dependence as a super-script, e.g. $c^{\chi,\sigma}$.

We now improve the bounds of \Cref{l:speed_upper_bound} to show that $c \leq 2$ for all $a$ sufficiently large, from which we conclude that $c_U =2$ (see \Cref{thm:existence_and_speed_lower}).   It is clear that \Cref{thm:slow} follows directly from the \Cref{p:slow} (below) after taking $a\to\infty$.  We state and prove \Cref{p:slow} here, postponing the proof of the main technical lemma until \Cref{sec:non_negative_eigenvalue}.
\begin{prop}\label{p:slow}
	Suppose $\theta < \theta_0$.  There exists $\delta_\theta>0$, depending only on $\theta$ and $K$ such that if $|\chi|(\sigma^2+\sigma^{-1}) \in (0,\delta_\theta)$ and $a> 1/\delta_\theta$, then
	\[
		c \leq 2.
	\]
\end{prop}
\begin{proof}
We proceed by contradiction, writing
\begin{equation}\label{e.c52102}
	c = 2 + \epsilon
		\quad\text{ for some }
		\e\in(0, 2 |\chi|(1 + \sigma^{-1})],
\end{equation}
where the upper bound follows from \Cref{l:speed_upper_bound}.  Define $w:[-a,a] \to \R$ by
\[
	u(x) = \exp\left\{-\frac{c}{2} x + \frac{1}{2}\int_0^x \tilde v(y)\, dy\right\} w(x).
\]
Let $V(x) = 1- (u + (c-\tilde v)^2/4 - \tilde v_x/2)$.  From~\eqref{e:slab}, we find
\begin{equation}\label{e:tw4}
	\begin{cases}
		-w_{xx} - w V =0 \qquad \text{ on } (-a,a),\\
		w(-a) = \exp\left\{-\frac{1}{2} ca - \frac{1}{2} \int_0^{-a} \tilde v(y)\, dy\right\},
			\quad w(0) = \theta,
			\quad w(a) = 0.
	\end{cases}
\end{equation}

At this point it is useful to notice a few key properties of $V$.  First,
\begin{equation}\label{e:tw5}
	V = - u - \e \left(1 + \frac{\e}{4}\right) + \tilde v\left(\frac{c}{2} - \frac{\tilde v}{4}\right) + \frac{\tilde v_x}{2}.
\end{equation}
The important thing is that, when $|\chi|(1+\sigma^{-1})$ is small enough,  $V< 0$ except (perhaps) far to the right where it is small and decays to zero.  This heuristic should imply that the principal eigenvalue $\lambda$ of $-\partial_{xx} - V$ should be positive.  The corresponding principal eigenfunction is then a supersolution of~\eqref{e:tw4}, which we use to derive a contradiction.

We now make this argument precise.  Let $(\lambda,\phi)$ solve
\begin{equation}\label{e:tw_eigenvalue}
	\begin{cases}
		-\phi_{xx} - \phi V = \lambda \phi \qquad \text{ in } (-a,a),\\
		1 = \phi(-a) = \phi(a), \quad \phi>0.
	\end{cases}
\end{equation}
One way to construct a solution of~\eqref{e:tw_eigenvalue} is via the Rayleigh quotient representation of $\lambda$; that is,
	\begin{equation}\label{e:Rayleigh}
		\lambda = \min_{\psi \in H_{\rm per}^1([-a,a])} \frac{\int_{-a}^a (\psi_x^2 - V \psi^2)\, dx}{\int_{-a}^a \psi^2 dx},
	\end{equation}
	where $H^1_{\rm per}([-a,a])$ is the closure of all $2a$ periodic $C^1$ functions under the $H^1([-a,a])$ norm.  The existence and positivity of a minimizer and the fact that a minimizer is, up to normalization, a solution of~\eqref{e:tw_eigenvalue} is classical.

We postpone the analysis of~\eqref{e:tw_eigenvalue} momentarily, though we record the important properties of $(\lambda,\phi)$ in~\Cref{l:eigenvalue}, below, which is proved in \Cref{sec:non_negative_eigenvalue}.
\begin{lem}\label{l:eigenvalue}
	For $|\chi|(\sigma^2 + \sigma^{-1})$ and $\theta_0$ small enough, $\lambda \geq 0$ and $\phi(0) \leq \exp\{ a/ 2\}$ for all $a$ sufficiently large when $c = 2+\eps$ with $\eps\geq0$.
\end{lem}

We now use \Cref{l:eigenvalue} to conclude.  Let $\phi_A(x) = A \phi(x)$.  If $A$ is sufficiently large, then $\phi_A > w$ because $w$ is bounded and $\phi$ is uniformly positive.  Let $A_0 = \inf\{A>0 : \phi_A >w\}$.  Then we find $x_0 \in [-a,a]$ such that $\phi_{A_0}(x_0) = w(x_0)$.  Since $w > 0$ on $[-a,a)$ and $\phi >0$, it follows that $A_0 > 0$.  
There are three cases.

{\bf Case one:  $x_0 = a$.}  This, clearly, cannot occur since $\phi_{A_0}(a) = A_0 \phi(a) > 0 =w(a)$.

{\bf Case two:  $x_0 \in (-a,a)$.}  First, note that, due to \Cref{l:eigenvalue},
\[
	-\phi_{xx}
		- V \phi
		= \lambda \phi
		\geq 0.
\]
Hence $\phi_{A_0}$ is a supersolution of~\eqref{e:tw_eigenvalue}.  Recall that $\phi_{A_0} > w(a)$ and $\phi_{A_0} \geq w$ on $[-a,a)$, by construction.  Applying the strong maximum principle, we conclude that $\phi_{A_0} > w$ on $(-a,a)$, which contradicts the fact that $x_0 \in (-a,a)$. Thus this case cannot occur.

{\bf Case three:  $x_0 = -a$.}  Then we have that
\[
	A_0 = \phi_{A_0}(-a) = w(-a) = \exp\left\{-\frac{1}{2}\left(ca - \int_0^{-a} \tilde v(y)\, dy\right)\right\}.
\]
In addition, using this identity for $A_0$, we find
\be\label{e.c6242}
	\theta
		= w(0)
		\leq \phi_{A_0}(0)
		\leq A_0 \phi(0)
		\leq \exp\left\{-\frac{1}{2}\left(ca - \int_0^{-a} \tilde v(y)\, dy\right)\right\} \exp\left\{\frac{a}{2}\right\}.
\ee
By assumption $c\geq 2$.  In addition, if $|\chi|(1+\sigma^{-1})$ is sufficiently small, then, by~\eqref{e:v_bound},
\[
	\frac{1}{2} \int_0^{-a} |\tilde v(y)|\, dy
		\leq \frac{a}{4}.
\]
Hence, the right hand side of~\eqref{e.c6242} clearly tends to zero as $a \to \infty$. 
This contradicts the fact that $\theta$ is positive and independent of $a$.

We have reached a contradiction in all cases, which implies that $c \leq 2$.  This concludes the proof.
\end{proof}

\subsubsection{Nonnegativity of the eigenvalue and asymptotics of its eigenfunction}\label{sec:non_negative_eigenvalue}

The first step is establishing the exponential decay of $u$ for all $a$ sufficiently large and all $\chi$ and $\sigma$ such that $|\chi|(1+\sigma^{-1})$ is sufficiently small.

\begin{lem}[Exponential decay of $u$]\label{l:exponential_decay}
	There exists $\mu, \delta_\theta>0$ such that if $|\chi|(1+\sigma^{-1}) \leq \delta_\theta$ and $a > 1/\delta_\theta$ and $\theta <1/4$, then
	\[
		u(x) \leq \theta e^{-\mu x} \qquad \text{ for all } x\geq 0.
	\]
	The constant $\delta_\theta$ depends on $\theta$.
\end{lem}
\begin{proof}
	Let $\rho = u_x / u$, and fix $\mu\in(0, 1/4]$ a constant to be chosen.  We prove this lemma in two steps.  First, we show that $\rho$ cannot have a maximum greater than $-\mu$ on $(0,a)$.  Since $\rho(a) = -\infty$, then $\rho(x) \leq \max\{-\mu, \rho(0)\}$.  Second, we show that $\rho(0) \leq -\mu$.  We conclude that $\rho(x) \leq -\mu$, which finishes the proof.
	
	{\bf Step one: $\rho$ cannot have an interior max larger than $-\mu$.}  Suppose that $x_0$ is the location of an interior maximum $M \geq -\mu$.  Using that $u \rho = u_x$ and $u\rho_x + u \rho^2 = u_{xx}$ in~\eqref{e:slab} (with $\tau =1$) and that $\rho = M$ and $\rho_x = 0$ at $x_0$, we obtain
	\[
		0
			= (c-\tilde v)u_x + u_{xx} + u(1-u-\tilde v_x)
			= u\left((c-\tilde v)M + M^2 + (1-u-\tilde v_x)\right).
	\]
	Hence,
	\be\label{e.c6241}
		M
			= -\frac{c - \tilde v \pm \sqrt{ (c-\tilde v)^2 - 4(1-u-\tilde v_x)}}{2}
			\leq \frac{-\left(c - \tilde v\right) + \sqrt{ (c-\tilde v)^2 - 4(1-u-\tilde v_x)}}{2}
				.
	\ee
	Recall from \Cref{thm:existence_and_speed_lower} that
	\be\label{e.c51601}
		2 \leq c \leq 2 + \frac{|\chi|}{2} + \frac{|\chi|}{\sigma}.
	\ee
	Hence, taking 
	$|\chi|(1+\sigma^{-1})$ to zero, the right hand side above tends to 
	\[
		\frac{-2 + \sqrt{4 - 4(1-u)}}{2}
			= -1+\sqrt{u(x_0)}
			\leq -1 + \sqrt{\theta}.
	\]
	As $\theta < 1/4$ by assumption, this yields $M < -1/2 < -1/4 \leq - \mu$.  Hence, if $|\chi|(1+\sigma^{-1})$ is sufficiently small,~\eqref{e.c6241} cannot hold.  This contradiction concludes the proof of this step.

	{\bf Step two: $\rho(0) \leq -\mu$.} 
	We argue by contradiction: if the claim were not true for all $a$ sufficiently large and $|\chi|(1+\sigma^{-1})$ sufficiently small, then there would exist sequences $a_n$, $\chi_n$, and $\sigma_n$, such that $a_n\to\infty$ and $|\chi_n|(1+\sigma_n^{-1}) \to 0$ as $n\to\infty$ such that $\rho^{\chi_n,\sigma_n}(0) > -\mu$.
	
	Recall~\eqref{e.c51601}. By compactness, $(c^{\chi_n,\sigma_n},u^{\chi_n,\sigma_n})$ converges to the minimal speed traveling wave solution $(2,U)$ of the Fisher-KPP equation
	\[
		-2 U_x = U_{xx} + U(1-U),
	\]
	with the normalization $U(0) = \theta$, and $\rho^{\chi_n,\sigma_n}(0)$ converges to $U_x(0)/U(0)$.
	
	Recall, from, e.g.,~\cite[Theorem~1.3]{Hamel}, the asymptotics of the Fisher-KPP equation: if $\overline U$ is the solution to the Fisher-KPP equation with the normalization $\overline U(0) = 1/2$, then there exists $\kappa$ such that $\overline U(x) / (\kappa x e^{-x}) \to 1$ as $x\to\infty$.  By the uniqueness of traveling wave solutions for the Fisher-KPP equation, there exists $x_\theta$, tending to $\infty$ as $\theta$ tends to zero, such that $U(x) = \overline U(x+x_\theta)$.  Further, by the monotonicity of the traveling wave for the Fisher-KPP equation, $x_\theta$ increases as $\theta$ decreases.  Hence,
	\[
		\sup_{[0,\infty)} \frac{U_x}{U}
			= \sup_{[x_\theta,\infty)} \frac{\overline U_x}{\overline U}
			\leq \sup_{[0,\infty)} \frac{\overline U_x}{\overline U}.
	\]
	The first inequality follows since $\theta < \theta_0 < 1/2$.
	
	Thus, let
	\[
		\mu =  - \frac{1}{4}\sup_{[0,\infty)} \frac{\overline U_x}{\overline U}.
	\]
	First, notice that $\overline U_x / \overline U \to -1$ as $x\to\infty$, so $\mu \leq 1/4$, which is consistent with our choice in the previous step.  On the other hand, since $\overline U_x/\overline U$ is a continuous, negative function on $[0,\infty)$ such that $\overline U_x(0) < 0$ and $\overline U_x / \overline U \to -1$ as $x\to\infty$, we find that $\mu > 0$.

	Choosing $|\chi|(1+\sigma^{-1})$ sufficiently small, depending on $\theta$, we find
	\[
		-\mu 
			< \rho(0)
			\leq \frac{1}{2} U(0)
			= \frac{1}{2} \sup_{[0,\infty)} \frac{U_x}{U}
			\leq - \mu.
	\]
	This is a contradiction, which concludes the proof.
%
%
\end{proof}

The next step is to establish a Poincar\'e type inequality. 
This is the key estimate in the proof of \Cref{l:eigenvalue}.
%

\begin{lem}\label{l:poincare}
	There exists a constant $C_K$, depending only on $K$, such that
	\begin{align}
		\int_{-a}^a \tilde v_x f^2 dx
			&\leq \frac{C_K |\chi| (1+ \sigma^2)}{\theta} \left( \int_{-a}^a u f^2 dx + \int_{-a}^a |f_x|^2 dx\right),
			\label{e:poincare1}\\
		\int_{-a}^a |\tilde v| f^2 dx
			&\leq \frac{C_K |\chi| (1+ \sigma^2)}{\theta} \left( \int_{-a}^a u f^2 dx + \int_{-a}^a |f_x|^2 dx\right)
			\label{e:poincare2}
	\end{align}
	for all $f \in H_{\rm per}^1([-a,a])$.
\end{lem}

We postpone the proof of this lemma momentarily and show its application in the proof of \Cref{l:eigenvalue}.
\begin{proof}[Proof of \Cref{l:eigenvalue}]
	Let $(\lambda, \phi)$ solve~\eqref{e:tw_eigenvalue}.  Then, using \Cref{l:poincare} and the form of $V$ (see~\eqref{e:tw5}), we find a universal constant $C$ such that
	\begin{equation}\label{e:tw10}
	\begin{split}
		-\int V\phi^2 dx
			&= \int u \phi^2 dx
				+ \int \e\left(1 + \frac{\e}{4}\right) \phi^2 dx
				- \int \tilde v \left( \frac{c}{2}- 
					\frac{\tilde v}{4}\right) \phi^2 dx
				- \int \frac{\tilde v_x}{2} \phi^2 dx\\
			&\geq \int u \phi^2 dx
				- C \frac{C_K |\chi| (1+ \sigma^2)}{\theta} \left(\int u \phi^2 dx
				+ \int\phi_x^2 dx\right).
	\end{split}
	\end{equation}
	Multiplying~\eqref{e:tw_eigenvalue} by $\phi$ and integrating by parts yields
	\[
		\lambda \int \phi^2 dx
			= \int \phi_x^2 dx - \int V \phi^2 dx.
	\]
	This, along with~\eqref{e:tw10}, implies that if $|\chi|$ is sufficiently small, then
	\[
		\lambda \int \phi^2 dx
			\geq \int \phi_x^2 dx + \int u \phi^2 dx - C \frac{C_K |\chi| (1+ \sigma^2)}{\theta} \left( \int u \phi^2 dx + \int \phi_x^2 dx \right)
			\geq 0.
	\]
	Hence, up to decreasing $|\chi|$, $\lambda \geq 0$, which finishes the first claim in \Cref{l:eigenvalue}.

	Before we examine the bound on $\phi(0)$, we require an upper bound on $\lambda$. To this end, we construct a test function for~\eqref{e:Rayleigh}; let
	\[
		\tilde \psi(x) =
			\begin{cases}
				A\left(x- \frac{a}{2}\right) \qquad &\text{ if } x \in \left(\frac{a}{2}, \frac{3a}{4}\right),\\
				\frac{Aa}{4} - A\left(x - \frac{3a}{4}\right) \qquad &\text{ if } x \in \left( \frac{3a}{4}, a\right)\\
				0 &\text{ otherwise},
			\end{cases}
	\]
	where $A = (96/a^3)^{1/2}$.  See \Cref{f1} for a cartoon of $\tilde \psi$.  \begin{figure}
\begin{center}
\begin{overpic}[scale=.75]
	{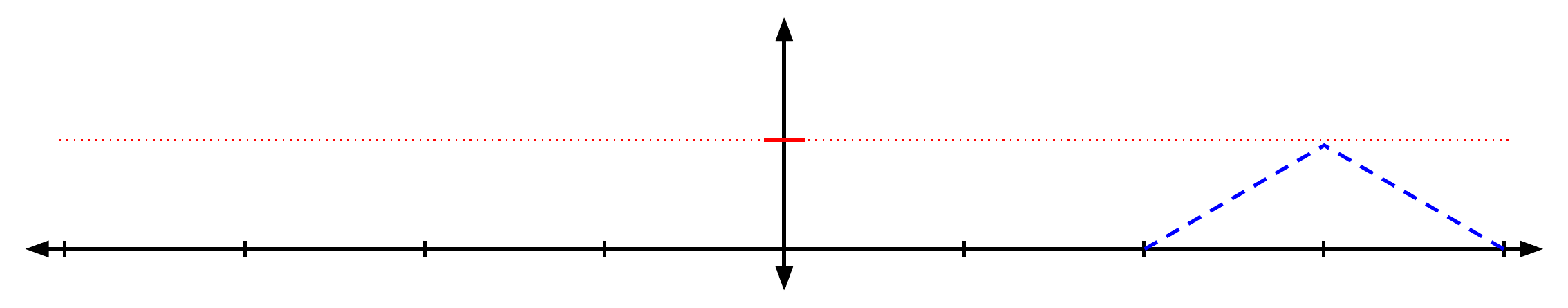}
	\put(2,.8){$-a$}
	\put(95.4, .8){$a$}
	\put(84, 13){\color{blue} $\tilde\psi$}
	\put(72.3,.8){$\frac{a}{2}$}
	\put(83.4,.8){$\frac{3a}{4}$}
	\put(50.5,12){\color{red} $\frac{Aa}{4} = \sqrt{\frac{6}{a}}$}
\end{overpic}
\caption{A cartoon of $\tilde\psi$, which is represented by the dashed line.}
\label{f1}
\end{center}
\end{figure}
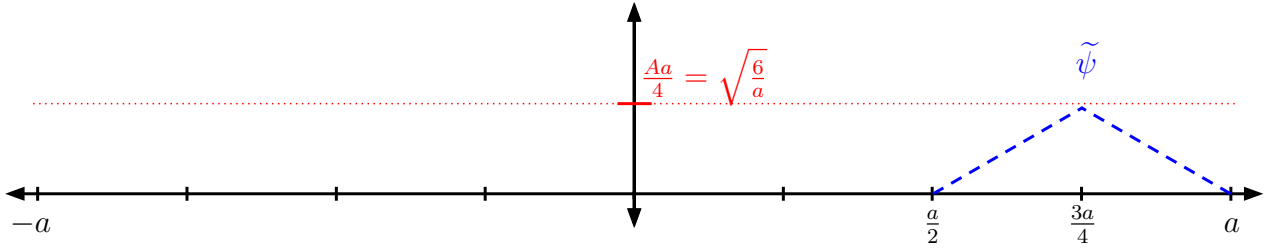
	We point out that
	\[
		\int \tilde \psi^2\, dx
			= \frac{A^2 a^3}{96} = 1
			\quad \text{ and } \quad
		\int \tilde \psi_x^2\, dx
			= A^2 \frac{a}{2} = \frac{48}{a}.
	\]
	
	Using \Cref{l:exponential_decay} and~\eqref{e:v_bound}, there is $C>0$ such that  $\max_{x \in [a/2,a]} (-V)\leq C|\chi|(1+\sigma^{-1})$ for $a$ sufficiently large.  Using $\tilde \psi$ as a test function in~\eqref{e:Rayleigh} along with these three estimates, yields
	\begin{equation}\label{e:lambda_above}
		\lambda 
			\leq C |\chi|(1 + \sigma^{-1}),
	\end{equation}
	when $a$ is chosen sufficiently large.  

	We now obtain the growth bound on $\phi$.  Let
	\begin{equation}\label{e.c52101}
		\delta = 2\max_{x\in [0,a]}\left(-\lambda - V(x)\right)
			= 2\max_{x\in [0,a]}\left(-\lambda + u + \e\left(1 + \frac{\e}{4}\right) - \tilde v \left( \frac{c}{2} - \frac{\tilde v}{4}\right) - \frac{\tilde v_x}{2}\right).
	\end{equation}
	Recall, from~\eqref{e:v_bound} that $|\tilde v|, |\tilde v_x| \leq |\chi|(1 + \sigma^{-1})$.  Using these and~\eqref{e:lambda_above}  while evaluating the right hand side of~\eqref{e.c52101}, we find
	\[
		\delta \geq 2u(0) - C |\chi|\left(1 + \frac{1}{\sigma}\right)
			= 2\theta - C |\chi|\left(1 + \frac{1}{\sigma}\right).
	\]
	Hence, after decreasing $|\chi|(1 + \sigma^{-1})$, we see that $\delta >0$.  Moreover, after using the monotonicity of $u$ (\Cref{l:monotonicity}) as well as the bound on $\eps$~\eqref{e.c52102}, we obtain a matching upper bound:
	\[
		\delta
			\leq 2\left(\max_{x\in[0,a]} u\right)
				+ C |\chi|\left(1 + \frac{1}{\sigma}\right)
			= 2 u(0) + C |\chi|\left(1 + \frac{1}{\sigma}\right)
			= 2 \theta + C |\chi|\left(1 + \frac{1}{\sigma}\right).
	\]
	It follows that 
%
%
	there exists a universal constant $C$ such that
	\[
		|\delta - 2\theta| < C|\chi|\left( 1 + \frac{1}{\sigma}\right).
	\]
	For the remainder of the proof, we assume that $\theta$ and $|\chi|(1+\sigma^{-1})$ are sufficiently small so that $2\sqrt \delta < 1/2$ and $\delta > \theta/2$ (recall that $\theta < \theta_0 < 1/100$).

	We now claim that $\phi(0) \leq e^{2 \sqrt \delta a}$, which would complete the proof.  We proceed by contradiction, assuming that $\phi(0) > e^{2 \sqrt \delta a}$.  For any $A> 0$, let
	\[
		\underline \phi_A(x)
			= A\left(
				\left(e^{2\sqrt \delta a} -  \frac{e^{3 \sqrt \delta a} -1}{e^{\sqrt \delta a} - e^{-  \sqrt \delta a}}\right)e^{\sqrt{\delta} x}
				 + \left(\frac{e^{3 \sqrt \delta a} - 1}{e^{ \sqrt \delta a} - e^{-\sqrt \delta a}}\right)e^{-\sqrt \delta x}\right).
	\]
	Importantly, we have that $\underline\phi_A(a) = A$, $\underline \phi_A(0) = A e^{2\sqrt{\delta} a}$, and 
	\[
		-(\underline \phi_{A})_{xx}
			= -\delta \underline \phi_A
			< (\lambda+ V) \underline\phi_A.
	\]
	We let $A_0 = \sup\{A : \underline \phi_A < \phi$ on $[0,a]\}$.  Continuity implies that $\phi$ is positive and bounded away from 0 and, hence, $A_0$ is well-defined and positive.  By continuity, there exists $x_0 \in [0,a]$ such that $\underline \phi_{A_0}(x_0) = \phi(x_0)>0$.  There are three cases to consider.
	
	{\bf Case one: $x_0 \in (0,a)$.}  Then, $x_0$ is the location of a minimum of $\phi - \underline \phi_{A_0}$, which implies that $\phi_{xx} \geq \underline\phi_{A_0,xx}$ at $x_0$.  Thus, at $x_0$,
	\[
		- \phi_{xx}
			\leq - (\underline\phi_{A_0})_{xx}
			< (\lambda + V) \underline \phi_{A_0}
			= (\lambda+V) \phi
			= - \phi_{xx},
	\]
	which is a contradiction.  Hence, $x_0 \notin (0,a)$.
	
	{\bf Case two: $x_0 = 0$.}  Then we have that
	\[
		\phi(0)
			= \underline \phi_{A_0}(0)
			= A_0 e^{2\sqrt \delta a}.
	\]
	By assumption, $\phi(0) > e^{2 \sqrt \delta a}$.  Hence, $A_0 > 1$.  On the other hand, by construction of $A_0$, we have that $\phi(a) \geq \underline \phi_{A_0}(a)$.  This yields a contradiction because $\phi(a) = 1$ and $\underline \phi_{A_0}(a) = A_0 > 1$.
	
	{\bf Case three: $x_0 = a$.}  This implies that
	\[
		1 = \phi(a)
			= \underline\phi_{A_0}(a)
			= A_0.
	\]
	Since $\phi(a) = \underline\phi_{A_0}(a)$ and $\phi(x) \geq \underline\phi_{A_0}(x)$ for all $x\in [0,a]$, we have that
	\begin{equation}\label{e:tw21}
	\begin{split}
		\phi_x(a)
			&\leq \underline\phi_{A_0,x}(a)
			= \sqrt\delta\left(e^{2\sqrt \delta a} -  \frac{e^{3 \sqrt \delta a} -1}{e^{\sqrt \delta a} - e^{- \sqrt \delta a}}\right)e^{\sqrt{\delta} a}
				 - \sqrt\delta \left(\frac{e^{3 \sqrt \delta a} - 1}{e^{ \sqrt \delta a} - e^{-\sqrt \delta a}}\right)e^{-\sqrt \delta a}\\
			&= \sqrt\delta\frac{- 2 e^{\sqrt \delta a} + 1+ e^{- 2 \sqrt \delta a}}{1 - e^{-  2\sqrt \delta a}}
			\leq - 2\sqrt\delta e^{\sqrt \delta a} + \sqrt \delta \frac{1+ e^{- 2 \sqrt \delta a}}{1 - e^{-  2\sqrt \delta a}}
			\leq - 2 \sqrt \delta e^{\sqrt \delta a} + 1.
	\end{split}
	\end{equation}
	where the last line follows when $a$ is sufficiently large.  
	The first equality comes from the definition of $\underline\phi_{A_0}$ and that $A_0 = 1$, while the second and third follow 
	from direct computations.
	
	We now obtain a contradiction by showing that $\phi_x(a)$ cannot take such a large value.  Before beginning we make a note about the regularity of $\phi_x$.  Note that $\phi_{xx} = -(\lambda + V) \phi \in L^2_{\rm per}([-a,a])$ and, hence, $\phi_x$ is continuous and periodic on $[-a,a]$.  As we use below, one consequence of this is that $\phi_x(-a) = \phi_x(a)$.

	Assume that $a$ is sufficiently large that, via~\eqref{e:tw21}, $\phi_x(a) < 0$.  Let 
	\[
		\bar x = \min\{x \geq -a : \phi_x(x) = \phi_x(a)/2\}.
	\]
	We show that $\bar x$ is well-defined.  Since $\phi$ is periodic,
	\[
		\int_{-a}^a \phi_x\, dx = 0.
	\]
	Hence, $\phi_x$ must be nonnegative somewhere.  Since $\phi_x(a) < 0$, it follows that $\{x: \phi_x(x) = \phi_x(a)/2\}$ is nonempty, which implies that $\bar x$ is well-defined.

		We note that
		\begin{equation}\label{e.c52103}
			\phi_x(x) \leq \frac{\phi_x(a)}{2}
				\quad\text{ for all } x \in[-a,\bar x].
		\end{equation}
		We see this as follows.  Since $\phi$ is periodic and $\phi_{xx}$ is bounded, then $\phi_x$ is continuous and periodic on $[-a,a]$.  This first yields that 
		\begin{equation}\label{e.c52104}
			\phi_x(-a) = \phi_x(a)<0.
		\end{equation}
		Using~\eqref{e.c52104} along with the continuity of $\phi_x$, the definition of $\bar x$, and the intermediate value theorem, we deduce~\eqref{e.c52103}.
	
	Next, we notice that there exists $\xi \in [-a, x_0]$ such that
	\be\label{e.c6253}
		\phi_{xx}(\xi)
			= \frac{ \phi_x(\bar x) - \phi_x(-a)}{\bar x - (-a)}
			= \frac{ \frac{\phi_x(a)}{2} - \phi_x(a)}{\bar x + a}
			= \frac{|\phi_x(a)|}{2(\bar x + a)}.
	\ee
	In order to show that $\phi_{xx}(\xi)$ is large, we obtain a smallness bound on $a+x_0$ through the positivity of $\phi$.  Indeed, by \eqref{e.c52104} and the mean value theorem,
	\[
		0 \leq \phi(\bar x)
			\leq \phi(-a) + \frac{\phi_x(a)}{2}(\bar x-(-a))
			= 1 - \frac{|\phi_x(a)|}{2}(\bar x + a),
	\]
	which implies that $\bar x + a \leq 2 / |\phi_x(a)|$.  Plugging this into~\eqref{e.c6253} yields
	\be\label{e.c6254}
		\phi_{xx}(\xi)
			\geq \left(\frac{\phi_x(a)}{2}\right)^2.
	\ee
	
	In addition, by~\eqref{e:tw_eigenvalue},~\eqref{e:lambda_above}, and the fact that $\phi(x) \leq 1$ for all $x\in [-a,\bar x]$ (which follows from~\eqref{e.c52104} and~\eqref{e:tw_eigenvalue})
	\be\label{e.c6255}
		|\phi_{xx}(\xi)|
			= |(V + \lambda)\phi|
			\leq C(1 + |\chi|(1+\sigma^{-1})).
	\ee
	
	Putting together~\eqref{e:tw21},~\eqref{e.c6254}, and~\eqref{e.c6255}, we find
	\[
		\left(- \sqrt \delta e^{\sqrt\delta a} + \frac{1}{2} \right)^2
			\leq C(1 + |\chi|(1+\sigma^{-1})).
	\]
	We have reached a contradiction since the right hand side is independent of $a$ while the left hand side tends to infinity as $a\to\infty$.  This concludes the proof.

\end{proof}

Now we prove the functional inequalities that were the crux of the argument that $\lambda$ is nonnegative.
\begin{proof}[Proof of \Cref{l:poincare}]
We begin by reducing~\eqref{e:poincare1} to~\eqref{e:poincare2}.  Indeed, integrating by parts, we obtain
\be\label{e.c6243}
	\int_{-a}^a \tilde v_x f^2\, dx
		= f(-a)^2\left[\tilde v(a) - \tilde v(-a)\right]
			- 2 \int_{-a}^a \tilde v f f_x\, dx.
\ee
Using the Sobolev inequality, we see that, for some $C>0$ that changes line-by-line in the sequel but is independent of all parameters,
\[
	f(-a)^2 [\tilde v(a) - \tilde v(-a)]
		\leq C\|\tilde v\|_{L^\infty}
			\left( \int_{-a}^{-a+1} (f^2 + f_x^2)\, dx\right)
\]
and using Young's inequality, we find
\[
	-2 \int_{-a}^a \tilde v f f_x dx
		\leq \int_{-a}^a |\tilde v| f^2 dx
			+ \|\tilde v\|_{L^\infty} \int_{-a}^a f_x^2 dx.
\]
Hence,~\eqref{e.c6243} becomes
\[
	\int_{-a}^a \tilde v_x f^2\, dx
		\leq C\|\tilde v\|_{L^\infty} \int_{-a}^{-a+1} f^2\, dx
			+ \|\tilde v\|_{L^\infty} \int_{-a}^a |f_x|^2 dx
			+ \int_{-a}^a |\tilde v| f^2 dx.
\]
Recalling the bounds on $\tilde v$ in~\eqref{e:v_bound} and that $u \geq \theta$ on $[-a,0]$ due to \Cref{p:slab}, we find
\begin{equation}\label{e:tw22}
\begin{split}
	\int_{-a}^a \tilde v_x f^2\, dx
		&\leq \frac{C|\chi|}{\theta} \int_{-a}^{-a+1} u f^2\, dx
			+ C|\chi| \int_{-a}^a f_x^2 dx
			+ \int_{-a}^a |\tilde v| f^2 dx\\
		&\leq \frac{C|\chi|}{\theta} \left( \int_{-a}^{a} u f^2\, dx + \int_{-a}^a |f_x|^2\, dx\right)
			+ \int_{-a}^a |\tilde v| f^2\,  dx.
\end{split}
\end{equation}

It is now clear from~\eqref{e:tw22} that~\eqref{e:poincare1} holds if~\eqref{e:poincare2} does.  We now establish~\eqref{e:poincare2}.  First, by \Cref{p:slab}, $u(x) \geq \theta$ for all $x\in[-a,0]$ so that $|\tilde v(x)|\leq C|\chi|u(x)/\theta$ for all $x \in [-a,0]$.  Thus,
\begin{equation}\label{e:tw23}
	\int_{-a}^{0} |\tilde v| f^2\, dx
		\leq \frac{C|\chi|}{\theta} \int_{-a}^{0} u f^2\, dx.
\end{equation}

Next, we bound the integral on $[0,a]$.  Indeed,
\begin{equation}\label{e:tw24}
\begin{split}
	\int_0^a |\tilde v| f^2\, dx
		&= \int_{0}^a |\tilde v| \left( \int_{0}^x f_x dy + f(0)\right)^2 dx
		\leq 2 \int_{0}^a |\tilde v| \left( x \int_{0}^x f_x^2 dy + f(0)^2\right)\, dx.
\end{split}
\end{equation}
As above, we bound $f(0)$ using the Sobolev inequality and the fact that $u \geq \theta$ on $[-1,0]$.  Indeed:
\[\begin{split}
	|f(0)|^2
		&\leq C \int_{-1}^0 f_x^2\, dx + C\int_{-1}^0 f^2\, dx
		\leq C \int_{-1}^0 f_x^2\, dx + \frac{C}{\theta}\int_{-1}^0 uf^2\, dx\\
		&\leq C \int_{-a}^a f_x^2\, dx + \frac{C}{\theta} \int_{-a}^a u f^2\, dx.
\end{split}\]
Hence,~\eqref{e:tw24} becomes
\[
	\int_{0}^a |\tilde v| f^2\, dx
		\leq C \left(\int_{-a}^a |f_x|^2\, dx + \frac{C}{\theta}\int_{-a}^a uf^2\, dx \right) \int_{0}^a (1+x) |\tilde v|\, dx.
\]

Clearly we have that
\[
	\int_0^a (1 + x) |\tilde v|\, dx
		\leq 2 \int_0^a x |\tilde v|\, dx
			+ \int_0^1 |\tilde v|\, dx
		\leq 2 \int_0^a x |\tilde v|\, dx
			+ C|\chi|.
\]
As a result, the proof is finished if
\begin{equation}\label{e:tw25}
	\int_0^a x |\tilde v|\, dx
		\leq C|\chi| \sigma^2.
\end{equation}
We now establish this inequality.  We begin by using the exact form of $\tilde v$ along with the decay of $u$ given in \Cref{l:exponential_decay}.  Indeed,
\begin{equation}\label{e:tw26}
\begin{split}
	\frac{1}{|\chi|}\int_0^a x |\tilde v|\, dx
		&\leq \int_0^a x \int_0^\infty |K_\sigma(y)||\tilde u(x-y) - \tilde u(x+y)|\, dy dx\\
		&= \int_0^a x \left(\int_0^{x/2} + \int_{x/2}^\infty\right) |K_\sigma(y)||\tilde u(x-y) - \tilde u(x+y)|\, dy dx\\
		&\leq \theta \int_0^a x \int_0^{x/2} |K_\sigma(y)| e^{-\mu(x-y)}\, dy dx
			+ 2\int_0^a x \int_{x/2}^\infty |K_\sigma(y)|\, dy dx.
\end{split}
\end{equation}
In the last inequality, we used that $u$ is bounded above by $2$ (recall \Cref{l:upper_bound_a} and that $|\chi|/\sigma < 1/2$, by assumption).

We bound the first term on the last line of~\eqref{e:tw26}.  Re-writing the integral, we find
\begin{equation}\label{e:tw27}
\begin{split}
	\int_0^a x\int_0^{x/2} |K_\sigma(y)| e^{-\mu(x-y)}\, dy dx
		&= \int_0^{a/2} |K_\sigma(y)| \int_{2y}^a x e^{-\mu(x-y)}\, dx dy\\
		&\leq \int_0^{a/2} |K_\sigma(y)| \left(\frac{e^{-\mu y}(1 + 2\mu y)}{\mu^2}\right)\, dy
		\leq C,
\end{split}
\end{equation}
where the last inequality follows from the fact that $\|K_\sigma\|_{L^1} = \|K\|_{L^1} = 1$ (see~\eqref{e:K_normalization}).

Next we bound the second term on the last line of~\eqref{e:tw26}.  Using the fact that $|K| = - K = - \overline K'$ on $[0,\infty)$ with $(1+|x|)\overline K \in L^1$ (see~\eqref{e:K_antiderivative}), we find
\begin{equation}\label{e:tw28}
\begin{split}
	\int_0^a x &\int_{x/2}^\infty |K_\sigma(y)|\, dy dx
		= -\int_0^a x \int_{x/2}^\infty \frac{1}{\sigma} \overline K'\left(\frac{y}{\sigma}\right)\, dy dx
		= -\int_0^a x \int_{x/2}^\infty \left(\overline K\left(\frac{y}{\sigma}\right) \right)_y dy dx\\
		&= \int_0^a x \overline K\left(\frac{x}{2\sigma}\right)\, dx
		\leq C\sigma^2.
\end{split}
\end{equation}
The combination of~\eqref{e:tw26}, \eqref{e:tw27}, and~\eqref{e:tw28} establishes~\eqref{e:tw25}.    Thus, the proof is complete.
\end{proof}

\section{Fast traveling waves: \Cref{thm:fast}}

We now establish large lower bounds on the traveling wave speed $c$ whenever $\sigma$ is large, in absolute terms and relative to $-\chi$.  

%
%
%
%
%
%
%
%
%


\begin{proof}[Proof of \Cref{thm:fast}]
	Fix $\epsilon\in (0,1/2)$ and suppose that
	\begin{equation}\label{e:slow_speed}
		c < \frac{(1 - \e)|\chi|}{2}.
	\end{equation}
	Recall that
	\[
		\overline K(x) = -\int_x^\infty K(y)\, dy
			= -\int_{x\sigma}^\infty K_\sigma(y)\, dy
			\quad\text{and}\quad
		\overline K(0) = \frac{1}{2}.
	\]
	 The last equality follows from the symmetry of $K$ and~\eqref{e:K_normalization}.
Now we define the constants
	\begin{equation}\label{e:constants1}
		\theta = \frac{|\chi|}{2R}
			\quad \text{ and }\quad
		R = \frac{1}{2} \overline K^{-1}\left( \frac{1-\frac{\e}{4}}{2}\right).
	\end{equation}
	Without loss of generality, we may restrict to the case where $\eps < 1/10$ and is sufficiently small that $R\leq 1$.  Next, choose $A_\eps$ such that, if  $\sigma, \sigma/|\chi| \geq A_\epsilon$, then
	\begin{equation}\label{e:constants2}
		\frac{\theta}{\sigma} = \frac{|\chi|}{2\sigma R} < \e/4
			\qquad \text{ and } \qquad
		\e > \frac{1}{R \sigma}.
	\end{equation}
	
	The proof proceeds by showing that either the front is ``stretched'' by the advective term, which results in a large speed, or the advection is large on a large interval.  To this end, we define ``ends'' of the front
	\[\begin{split}
		&x_1 = \max\{y  \in \R : u(y) = 1 - \theta \sigma^{-1}\} \quad \text{ and}\\
		&x_2 = \max\{y \in \R : u(y) = \theta \sigma^{-1}\},
	\end{split}\]
	which are well-defined due to \Cref{l:one}.  Due to \Cref{l:monotonicity}, we see that $u$ is decreasing on
	$[x_1, \infty)$
	because, by the choice of $\theta$, $1-\theta/\sigma < (1 - \chi/2\sigma)^{-1}$ (recall that $R\leq 1$).  
	We conclude that
	\begin{equation}\label{e:u_x_1_x_2}
		u(x) \in \begin{cases}
			\left[ 1 - \frac{\theta}{\sigma}, 1\right) \qquad &\text{ if } x\leq x_1\\
			\left[\frac{\theta}{\sigma}, 1 - \frac{\theta}{\sigma}\right] \qquad &\text{ if } x \in [x_1, x_2]\\
			\left( 0, \frac{\theta}{\sigma}\right] &\text{ if } x\geq x_2.
		\end{cases}
	\end{equation}
	We now separate into two cases.

	\medskip
	
	{\bf Case one: $x_2 - x_1 > R\sigma$.}  For any $L>0$, by integrating~\eqref{e:wave} from $-L$ to $L$, we find
	\[
		c[u(-L) - u(L)]
			+ [ v(L) u(L) - v(-L) u(-L)]
			= u_x(L) - u_x(-L) + \int_{-L}^L u(1-u)\, dx.
	\]
	The first term on the left and the last term on the right are the crucial terms.  Our goal is to show that all other terms tend to zero.  First, since $u(-L) \to 1$ and $u(L) \to 0$ as $L\to\infty$, it follows that
	\[
		c[u(-L) - u(L)]
			\to c
		\quad\text{ and }\quad
		v(\pm L) = \chi(K_\sigma * u)(\pm L) \to 0
		\qquad\text{ as } L \to\infty.
	\] 
	  Using elliptic regularity theory, since $u$ tends to constants at $\pm \infty$, then $u_x(\pm L) \to 0$ as $L \to \infty$.
	Putting all of these together, we find
	\[
		\begin{split}
		c
		&= \lim_{L\to\infty} c[u(-L) - u(L)]
		\\&
		= \lim_{L\to\infty} \left(- [ v(L) u(L) - v(-L) u(-L)]
			+ u_x(L) - u_x(-L) + \int_{-L}^L u(1-u)\right)
		\\&= \lim_{L\to\infty} \int_{-L}^L u(1-u)\, dx
		= \int_{-\infty}^\infty u(1-u)\, dx.
		\end{split}
	\]
	By \Cref{l:upper_bound}, we observe that $u(1-u) \geq 0$.  Hence,
	\[
		c \geq \int_{x_1}^{x_2} u(1-u)\, dx.
	\]
	
	We now obtain a lower bound on the integral term. By~\eqref{e:u_x_1_x_2} and basic calculus, 
	$u(1-u) \geq (\theta/\sigma)(1-\theta/\sigma)$ on $[x_1,x_2]$.  Thus, we find
	\[\begin{split}
		c
			&\geq \int_{x_1}^{x_2} u(1-u)\, dx
			\geq \left( 1 - \frac{\theta}{\sigma}\right) \frac{\theta}{\sigma} |x_2 - x_1|
			\geq \left( 1 - \frac{\theta}{\sigma}\right) \theta R
			\geq \left(1 - \frac{\e}{4}\right) \frac{|\chi|}{2},
	\end{split}\]
	which contradicts~\eqref{e:slow_speed}.  Hence, case one cannot occur.

	\medskip
	
	{\bf Case two: $x_2 - x_1 \leq R\sigma$.}  In this case, we build an explicit sub-solution in order to ``push'' the wave at a fast speed.  The key step in doing so is in estimating the effect of the advective term $v$ on an interval near the front.
	
	To obtain this estimate, we obtain a lower bound on $v$ near the front for any $x \in [x_2, x_2 + R\sigma]$.
	First, we note that $u(x+y) \leq u(x-y)$ for any $x \geq x_2$ 
	and $y \geq 0$.  Indeed, if $x-y > x_2$, this follows because $u$ is decreasing on $[x_2, \infty]$.  On the other hand, if $x-y \leq x_2$ then $u(x-y) \geq \theta/\sigma$ (see~\eqref{e:u_x_1_x_2}) and $u(x+y) \leq \theta/\sigma$ since $x + y \geq x_2$.

	As a consequence of this inequality, we find
\[
	\begin{split}
		v(x)
			= \chi(K_\sigma * u)(x)
			&= \chi\int_0^\infty K_\sigma(y) [u(x-y) - u(x+y)]\, dy\\
			& \geq \chi \int_{x-x_1}^\infty K_\sigma(y) [u(x-y) - u(x+y)]\, dy
				\quad\text{ for } x\geq x_2.
	\end{split}
\]
Above we used that $\chi K_\sigma \geq 0$ on the domain of integration.  
Then, using~\eqref{e:u_x_1_x_2} yields
\[
	v(x)
			\geq \chi \int_{x-x_1}^\infty K_\sigma(y) \left(1 - \frac{2\theta}{\sigma}\right)\, dy.
\]
Finally, observe that $x-x_1 = (x-x_2) + (x_2 - x_1) \leq 2R \sigma$ and, thus,
\begin{equation}
	\begin{split}
		v(x)
			&\geq \chi \int_{2R\sigma}^\infty K_\sigma(y) \left(1 - \frac{2\theta}{\sigma}\right)\, dy
			= |\chi| \left(1 - \frac{2\theta}{\sigma}\right) \overline K(2R).
	\end{split}
\end{equation}

We now use the choice of constants $\theta$ and $R$,~\eqref{e:constants1} and~\eqref{e:constants2}.  We conclude that, for all $x\in [x_2, x_2 + R\sigma]$,
\begin{equation}\label{e:advection_lower_bound}
	v(x)
		\geq \frac{|\chi|}{2} \left(1 - \frac{\e}{2}\right)^2
		\geq c.
\end{equation}
	The second inequality above follows from~\eqref{e:slow_speed}.

	Using~\eqref{e:advection_lower_bound}, we define our sub-solution.  For any $A>0$, let
	\begin{equation}\label{e:underline_u}
		\underline u_A(x) = A \sin\left(\frac{\pi(x - x_2)}{R\sigma}\right).
	\end{equation}
	Notice that $\underline u_A(x_2) = \underline u_A(x_2 + R\sigma) = 0$.
	
	By \Cref{def:wave}, $u>0$.  Then, by the compactness of $[x_2, x_2 + R\sigma]$ and the continuity of $u$ (due to elliptic regularity), $\inf_{[x_2,x_2+R\sigma]} u>0$.  Hence, $\underline u_A < u$ on $[x_2, x_2+R\sigma]$ if $A$ is sufficiently small.  We then ``raise'' $\underline u_A$ up until it touches $u$; i.e., let
	\[
		A_0 = \sup\{A : \underline u_A(x) < u(x) \text{ for } x \in [x_2, x_2 + R\sigma]\}.
	\]
	Thus, there exists $x_0 \in (x_2, x_2 + R\sigma)$ such that $\underline u_{A_0}(x_0) = u(x_0)$.  Define
	\[
		\phi(x) = u(x) - \underline u_{A_0}(x).
	\]
	Note that $\phi \geq 0$ on $[x_2, x_2 + R\sigma]$, $\phi(x_2) , \phi(x_2 + L \sigma) > 0$, and that $x_0$ is the location of a zero minimum of $\phi$, which yields
	\begin{equation}\label{e:phi}
		\phi(x_0) = 0,
		\quad \phi_x(x_0)  = 0,
			\quad \text{ and }\quad
		\phi_{xx} (x_0) \geq 0.
	\end{equation}
	Also, notice that, since $x_0 > x_2$, then $u(x_0) < \theta/\sigma < (1 - \chi/\sigma)^{-1}$, which yields $\underline u_{A_0,x}(x_0) = u_x(x_0) < 0$.
	
	We now derive a differential inequality for $\phi$.  Indeed, using~\eqref{e:wave},~\eqref{e:phi}, and that $\phi(x_0) = 0$, we find, at $x_0$,
	\begin{equation}\label{e:c12}
	\begin{split}
		0 \geq (v - c) \phi_x - &\phi_{xx} - (1 - v_x)\phi
			= - u^2 - (v-c) \underline u_{A_0, x} + \underline u_{A_0,xx} + \underline u_{A_0}(1 - v_x).
%
%
%
	\end{split}
	\end{equation}
	By~\eqref{e:advection_lower_bound}, $v- c \geq 0$ at $x_0$.  In addition, by \Cref{l:monotonicity} and the choice of $x_2$, $\underline u_{A_0,x}(x_0) = u_x(x_0) < 0$.  Thus,~\eqref{e:c12} becomes, at $x_0$,
	\begin{equation}\label{e:c13}
		0 \geq - u^2 + \underline u_{A_0,xx} + \underline u_{A_0}(1 - v_x).
	\end{equation}
	Since $x_0 > x_2$, $u(x_0) < \theta/ \sigma$.  Also, recall that $u(x_0) = \underline u_{A_0}(x_0)$ and, by~\eqref{e:v_x},
	\[
		v_x(x_0)
			= \frac{|\chi|}{\sigma}u(x_0) - |\chi| \int_0^\infty K_{\sigma, x}(y) (u(x_0 - y) + 
			u(x_0 + y)
			)\, dy.
	\]
	Hence, $-u(x_0)^2 > - \underline u_{A_0}(x_0)\theta/\sigma$ and $-v_x(x_0) > -|\chi|\theta/\sigma^2$.
	Using these two inequalities along with~\eqref{e:c13} yields, at $x_0$,
	\begin{equation}\label{e:c14}
		0 > \underline u_{A_0,xx} + \underline u_{A_0}\left(1 - \frac{\theta}{\sigma} - \frac{\theta |\chi|}{\sigma^2}\right).
	\end{equation}
	Using the explicit form of $\underline u_{A_0}$ above, $\underline u_{A_0,xx} = - (\pi/R\sigma)^2 \underline u_{A_0}$; hence,
	\[
		0 > \underline u_{A_0}(x_0) \left( 1 - \frac{\theta}{\sigma} - \frac{|\chi|\theta}{\sigma^2} - \frac{\pi^2}{R^2\sigma^2} \right).
	\]
	Recalling~\eqref{e:constants1} and~\eqref{e:constants2}, we obtain
		\[
		0 > \underline u_{A_0}(x_0) \left( 1 - \frac{\theta}{\sigma} - \frac{|\chi|\theta}{\sigma^2} - \frac{\pi^2}{R^2\sigma^2} \right)
			\geq \underline u_{A_0}(x_0) \left(1 - \frac{\eps}{4} - \frac{\eps^2 R}{8} - \pi^2\e^2\right) \geq 0,
	\]
	which is a contradiction (recall that $\eps<1/10$ and $R\leq 1$).  This completes the proof.
\end{proof}

\bibliographystyle{abbrv}
\bibliography{negative_chemo}
\end{document}